\documentclass[a4paper,oneside,10pt]{amsart}
\usepackage{cite} 
\usepackage{ dsfont }

\usepackage[text={420pt,660pt},centering]{geometry}
\usepackage{comment}
\usepackage{graphicx}
\usepackage{MnSymbol}
\usepackage{mathtools} 
\usepackage{stmaryrd} 
\usepackage[colorlinks=true, pdfstartview=FitV, linkcolor=blue, citecolor=blue, urlcolor=blue,pagebackref=false]{hyperref}
\usepackage{microtype}
\usepackage{xcolor}

\usepackage{enumitem}

\usepackage{float}

\definecolor{darkgreen}{rgb}{0,0.5,0}
\definecolor{darkblue}{rgb}{0,0,0.7}
\definecolor{darkred}{rgb}{0.9,0.1,0.1}

\newcommand{\hbc} [1]{{\color{orange}{#1}}}

\newcommand{\htc} [1]{{\color{violet}{#1}}}

\newcommand{\hdccomment}[1]{\marginpar{\raggedright\scriptsize{\textcolor{red}{#1}}}}

\usepackage{amsthm}

\newtheorem{proposition}{Proposition}
\newtheorem{theorem}[proposition]{Theorem}
\newtheorem{lemma}[proposition]{Lemma}
\newtheorem{corollary}[proposition]{Corollary}

\theoremstyle{remark}
\newtheorem{remark}[proposition]{Remark}

\theoremstyle{definition}
\newtheorem{definition}[proposition]{Definition}

\numberwithin{equation}{section}
\numberwithin{proposition}{section}
\numberwithin{figure}{section}
\numberwithin{table}{section}

\newcommand{\Z}{\mathbb{Z}}
\newcommand{\N}{\mathbb{N}}

\newcommand{\R}{\mathbb{R}}

\renewcommand{\P}{\mathbb{P}}

\newcommand{\ep}{\varepsilon}
\newcommand{\eps}{\varepsilon}

\renewcommand{\le}{\leqslant}
\renewcommand{\ge}{\geqslant}
\renewcommand{\leq}{\leqslant}
\renewcommand{\geq}{\geqslant}

\renewcommand{\subset}{\subseteq}

\newcommand{\Ll}{\left}
\newcommand{\Rr}{\right}
\renewcommand{\d}{\,\mathrm{d}}

\newcommand{\1}{\mathds{1}}

\newcommand{\mcl}{\mathcal}

\DeclareMathOperator{\supp}{supp}

\newcommand{\lp}{\left(}
\newcommand{\rp}{\right)}

\newcommand{\comple}{{\mathsf{c}}}

\newcommand{\Ann}{\mathrm{Ann}}

\newcommand{\itr}{\mathrm{int}} 

\newcommand{\bc}{\mathbf{c}}

\newcommand{\lo}{{\ell}} 
\newcommand{\rz}{z}

\newcommand{\calA}{\mathcal{A}}

\newcommand{\calF}{\mathcal{F}}
\newcommand{\calG}{\mathcal{G}}
\newcommand{\calH}{\mathcal{H}}
\newcommand{\calI}{\mathcal{I}}
\newcommand{\calJ}{\mathcal{J}}

\newcommand{\calL}{\mathcal{L}}

\newcommand{\calP}{\mathcal{P}}
\newcommand{\calQ}{\mathcal{Q}}

\newcommand{\bbE}{\mathbb{E}}

\newcommand{\bbP}{\mathbb{P}}

\newcommand{\bbR}{\mathbb{R}}

\newcommand{\bbZ}{\mathbb{Z}}

\def\LE{\!{\begin{array}{c}<\\[-8pt]\frown\end{array}}\!}
\def\GE{\!{\begin{array}{c}>\\[-8pt]\frown\end{array}}\!}

\newcommand{\ges}{\GE}
\newcommand{\les}{\LE}
\newcommand{\one}{\mathds{1}}

\newcommand{\Pin}{\calP^{\mathrm{in}}}
\newcommand{\Pout}{\calP^{\mathrm{out}}}
\newcommand{\Fin}{\calF_{\mathrm{in}}}
\newcommand{\Fout}{\calF_{\mathrm{out}}}

\newcommand{\lra}{\leftrightarrow}

\newcommand{\rcE}{\phi}
\newcommand{\rcP}{\phi}
\newcommand{\lpcfg}{\calL} 
\newcommand{\dst}{d} 

\usepackage{mathtools} \mathtoolsset{showonlyrefs} 

\begin{document}


\title[Critical exponents for planar random-cluster model with~$q=4$]{Critical exponents for planar random-cluster model with cluster-weight~$q=4$}


\author{Hong-Bin Chen}
\address[Hong-Bin Chen]{New York University Shanghai}
\email{\href{mailto: hbc236@nyu.edu}{hbc236@nyu.edu}}

\author{Hugo Duminil-Copin}
\address[Hugo Duminil-Copin]{Institut des Hautes \'Etudes Scientifiques \& Universit\'e de Gen\`eve}
\email{\href{mailto: hugo.duminil@unige.ch}{hugo.duminil@unige.ch}}

\author{Tiancheng He}
\address[Tiancheng He]{Universit\'e de Gen\`eve}
\email{\href{mailto: tiancheng.he@unige.ch}{tiancheng.he@unige.ch}}

\author{Fran\c cois Jacopin}
\address[Fran\c cois Jacopin]{LPSM, Sorbonne Universit\'e}
\email{\href{mailto: jacopin@ihes.fr}{jacopin@ihes.fr} 
}

\author{Dmitry Krachun}
\address[Dmitry Krachun]{Princeton University}
\email{\href{mailto: dk9781@princeton.edu}{dk9781@princeton.edu}}

\author{Ioan Manolescu}
\address[Ioan Manolescu]{Universit\'e de Fribourg}
\email{\href{mailto: ioan.manolescu@unifr.ch}{ioan.manolescu@unifr.ch}}

\author{Jiaming Xia}
\address[Jiaming Xia]{Shanghai Institute for Mathematics and Interdisciplinary Sciences}
\email{\href{mailto: xia@simis.cn}{xia@simis.cn}}

\begin{abstract}
Using the Baxter--Kelland--Wu coupling and the convergence of the height function of the six-vertex model to the Gaussian Free Field, 
we extract critical exponents for the planar critical random-cluster model at~$q=4$, and the planar four-state Potts model.
\end{abstract}

\maketitle

%
%
%
%
%
%

\section{Introduction}

\subsection{Motivation}
Since the 1980s, several strategies have been developed in physics to predict the critical exponents of planar models undergoing continuous phase transitions. A major breakthrough came with the work of Belavin, Polyakov, and Zamolodchikov, who introduced conformal field theory~\cite{belavin1984infinite,belavin1984infiniteJSP}. This framework exploits conformal invariance to describe, with remarkable precision, the scaling limits of a wide variety of two-dimensional lattice models at criticality. The conformal invariance approach has since flourished in mathematics, notably with the advent of the Schramm--Loewner Evolution in the early 2000s \cite{schramm2000scaling,lawler8dimension,lawler2001values,lawler2001valuesII,lawler2004conformal} and the first rigorous proofs of conformal invariance by Kenyon, Smirnov, and others \cite{kenyon2000conformal,smirnov2001critical,smirnov2010conformal,chelkak2012universality,kenyon2014conformal,duminil2021conformalI}.

In parallel, physicists in the 1980s proposed another line of attack. While still motivated by conformal invariance, it takes a different form: one seeks discrete mappings between integrable lattice models and height-function models, whose large-scale behavior is governed by the Coulomb Gas' system of charged particles in the plane interacting through the two-dimensional electrostatic potential. This approach translates certain correlation functions of the scaling limit of lattice models into Gaussian Free Field (GFF) correlations. 
It was pioneered by Nienhuis for loop~$O(n)$ models~\cite{Nienhuis82}, and has since been applied to a broad class of systems.

At the heart of this program lies the six-vertex (6V) model, a cornerstone of two-dimensional statistical mechanics. As a paradigmatic integrable system, its large-scale behavior is intimately connected to both the Coulomb Gas and the Gaussian Free Field. In a recent paper~\cite{DumKozLamMan26}, the full-plane scaling limit of the six-vertex model was identified over a substantial part of its phase diagram; see also~\cite{DumSidTas13,DumKarManOul20,DumManTas16b} for qualitative features of the model in this regime. This regime is directly linked to several important models, most notably the critical random-cluster model with cluster weight~$q\in[1,4]$, which in turn corresponds to the~$q$-state Potts model for~$q\in\{2,3,4\}$. The Baxter--Kelland--Wu (BKW) mapping provides a bridge between observables of the critical random-cluster model and exponential correlations of the six-vertex height function. The convergence result of~\cite{DumKozLamMan26} thus opens the door to deriving critical exponents for the random-cluster model, substantially broadening the range of models for which such exponents can be rigorously obtained.

In this paper, we concentrate on the case where the Baxter--Kelland--Wu mapping is most transparent and powerful: the random-cluster model with cluster weight~$q=4$ and the corresponding six-vertex model. In this setting, the mapping enables the derivation from the GFF scaling limit~\cite{DumKozLamMan26} of all critical exponents governing the behavior of key thermodynamic quantities of the random-cluster model, such as the infinite-cluster density, susceptibility, correlation length, and free energy. It also sheds light on the large-scale properties of the critical regime. Finally, by exploiting the Edwards--Sokal coupling between the~$q=4$ random-cluster model and the four-state Potts model, these results can be transferred to the Potts setting.

The same strategy may be attempted for general~$q \in [1,4]$, but raises additional challenges. 
Indeed, the results presented below are extended to an interval~$q \in (4-\eps,4]$ in \cite{qcloseto4} and partial results are obtained for all~$q \in [1,4]$ in \cite{alpha1, alpha2}.

\subsection{Main results for the random-cluster model}

Consider a finite sub-graph~$G=(V,E)$ of~$\Z^2$. A percolation configuration~$\omega$ on~$G$ is an element of~$\{0,1\}^E$. An edge~$e\in E$ is called \emph{open} if~$\omega_e=1$, and \emph{closed} otherwise. A configuration~$\omega$ can be seen as a subgraph of~$G$ with vertex-set~$V$ and edge-set~$\{e\in E:\omega_e=1\}$.

The boundary of~$G$, denoted~$\partial G$, is the set of vertices of~$G$ incident to less than four edges of~$G$. 
A \emph{boundary condition}~$\xi$ on~$G$ is a partition of~$\partial G$. Two vertices are said to be \emph{wired together} if they are in the same element of~$\xi$. 
Two boundary conditions play a special role: the wired ones, denoted~$1$, with all vertices wired together, 
and the free ones, denoted~$0$, with no wiring between boundary vertices. 

The \emph{random-cluster measure} on~$G$ with edge-weight~$p\in(0,1)$, cluster-weight~$q>0$, and boundary condition~$\xi$ is defined by
\begin{equation*}
    \rcP^\xi_{G,p,q}[\omega]
    :=\frac{1}{Z^\xi(G,p,q)} \left(\frac{p}{1-p}\right)^{|\omega|}q^{k\left(\omega^\xi\right)},
\end{equation*}
where~$|\omega|:=\sum_{e\in E}\omega_e$ is the number of open edges,~$k(\omega)$ is the number of connected components of the graph induced by~$\omega$,~$\omega^\xi$ is the graph obtained from~$\omega$ by identifying wired vertices together, and~$Z^\xi(G,p,q)$ is the \emph{partition function} ensuring that~$\rcP^\xi_{G,p,q}$ is a probability measure.

In this paper, we focus on the case~$q=4$ and drop it from the notation. One may take an infinite-volume limit of the measures~$\rcP^\xi_{G,p}$ by letting~$G$ tend to~$\mathbb Z^2$. When~$q=4$, the limit exists and is independent of the boundary conditions~$\xi$ \cite{DumSidTas13}. We denote the unique limiting measure by~$\rcP_{\mathbb Z^2,p}$. 
For more background on the random-cluster model, we direct the reader to \cite{grimmett2006randomcluster, duminil2017lectures}.

Let~$\mathbf C$ denote the connected component of the origin. 
The model undergoes a phase transition at
$ p_c:=\inf\{p\in[0,1]:\theta(p)>0\}$, where
 ~$\theta(p):=\rcP_{\mathbb Z^2,p}[|\mathbf C|=\infty]$.
It was proved in \cite{beffara2012self,beffara2015critical} that~$p_c=2/3$. 
This value also coincides with the point at which the 
\emph{susceptibility}
\[
    \chi(p):=\rcE_{\mathbb Z^2,p}[|\mathbf C|\cdot \mathbf 1_{\{|\mathbf C|<\infty\}}],
\]
and the \emph{correlation length}
\[
    \xi(p):=\Big(\lim_{n\to\infty}-\tfrac1n \log \rcP_{\mathbb Z^2,p}\big[(n,0)\in \mathbf C,\,|\mathbf C|<\infty\big]\Big)^{-1},
\]
diverge. Moreover,~$p_c$ is the unique value of~$p$ at which the \emph{free energy}
\[
    f(p):=\lim_{n\to\infty}-\tfrac1{|\Lambda_n|}\log Z^0(\Lambda_n,p)
\]
is not analytic \cite{DobShl87, MarOliSch94, vEnt+97}. 

Our first result describes the precise behaviour of these thermodynamic quantities near~$p_c$.

\begin{theorem}\label{cor:near critical}
For the random-cluster model on~$\mathbb Z^2$ with cluster-weight~$q=4$,
\mathtoolsset{showonlyrefs=false}%
\begin{align}
\label{eq:RCM_1}f''(p)&=|p-p_c|^{-2/3+o(1)} &\text{ as~$p\to p_c$},\\
\label{eq:RCM_2}\theta(p)&=(p-p_c)^{1/12+o(1)} &\text{ as~$p\searrow p_c$},\\
\label{eq:RCM_3}\chi(p)&=|p-p_c|^{-7/6+o(1)} &\text{ as~$p\to p_c$},\\
\label{eq:RCM_4}\xi(p)&=|p-p_c|^{-2/3+o(1)} &\text{ as~$p\to p_c$},\\
\label{eq:RCM_45}\rcP_{\mathbb Z^2,p_c}[0\lra x]&= |x|^{-1/4 + o(1)} &\text{ as~$x\to \infty$},\\
\label{eq:RCM_5}\rcP_{\mathbb Z^2,p_c}[|\mathbf C|\ge n]&=n^{-1/15+o(1)} &\text{ as~$n\to\infty$}.
\end{align}
\mathtoolsset{showonlyrefs=true}%
\end{theorem}

By the scaling relations of \cite{duminil2022planar}, understanding these quantities reduces to analyzing the asymptotics of certain probabilities at criticality. We introduce them next.
For subsets~$A,B\subset\mathbb R^2$, we say that~$A$ is \emph{connected} to~$B$, denoted~$A\longleftrightarrow B$, if there exists a path of open edges in~$\omega$ connecting~$A$ to~$B$. For~$N\in\N$, let~$\Lambda_N=[-N,N]^2$. Let~$e$ be an edge incident to the origin, and set
\begin{align*}
    \pi_1(N) &:= \rcP_{\Z^2,p_c}\left[0\longleftrightarrow \partial\Lambda_N\right],\\
    \Delta(N)&:=\rcP_{\Lambda_N,p_c}^1[\omega_e]-\rcP_{\Lambda_N,p_c}^0[\omega_e].
\end{align*}

The scaling relations of \cite{duminil2022planar} rephrase the quantities of Theorem~\ref{cor:near critical} in terms of~$\pi_1(N)$ and~$\Delta(N)$. 
The next theorem focuses on the asymptotic behavior of the latter two. 
For completeness, we also include another interesting critical quantity that can be studied in parallel to~$\Delta(N)$, and that evaluates the fractal dimension of long interfaces. 
We define it before stating our main result. 

Every configuration~$\omega$ on~$\mathbb Z^2$ induces a ``dual'' configuration~$\omega^*$ on the dual lattice~$(\mathbb Z^2)^*:=(\tfrac12,\tfrac12)+\mathbb Z^2$ by declaring an edge~$e^*$ of~$(\mathbb Z^2)^*$ open in~$\omega^*$ if and only if the ``primal'' edge~$e$ of~$\mathbb Z^2$ intersecting it in its middle is closed in~$\omega$. 
Moreover, if $\omega\sim \phi_{\mathbb Z^2,p_c}$, then $\omega^*$ has the same distribution as $\omega$ on the dual lattice. This property is called \emph{self-duality}.
For subsets~$A,B\subset\mathbb R^2$, we write~$A\stackrel{*}{\longleftrightarrow}B$ if there exists a path of open edges of~$\omega^*$ connecting~$A$ to~$B$. Define
\[
    \pi_2(N):=\rcP_{\Z^2,p_c}\left[0\longleftrightarrow \partial\Lambda_N,\;(\tfrac12,\tfrac12)\stackrel{*}{\longleftrightarrow}\Lambda_N^c\right].
\]

\begin{theorem}\label{thm:main}
For the critical random-cluster model on~$\mathbb Z^2$ with cluster-weight~$q=4$, as~$N$ tends to $\infty$, 
\begin{align}
\label{eq:a1} \pi_1(N) = N^{-1/8 + o(1)},\\
\label{eq:a2} \pi_2(N) = N^{-1/2 + o(1)},\\
\label{eq:a3} \Delta(N) = N^{-1/2 + o(1)}.
\end{align}
\end{theorem}

In fact, these results will be obtained from stronger scale-to-scale estimates, which also provide evidence of scale invariance. In order to keep a light introduction, we omit the precise statements here but still include an interesting result on the two-point connection probabilities of the critical model.

\begin{proposition}\label{prop:limit_two_pt}
There exists a function~$f:(0,1)\to(0,\infty)$ such that
\begin{align}
\label{eq:RCM_6}
\lim_{\delta\to 0} f(\delta)^{-2}\rcE_{\delta \Z^2,p_c,4}[0\longleftrightarrow x_\delta]=|x|^{-1/4}
\end{align}
uniformly for~$x$ in any compact subset of~$\R^2\setminus\{0\}$, where~$x_\delta$ denotes the vertex of~$\delta\Z^2$ closest to~$x$ (if there are more than one, choose one arbitrarily), and~$f(\delta)=\delta^{1/8+o(1)}$ as~$\delta$ tends to $0$.
\end{proposition}

Both the BKW correspondence and the GFF convergence of the associated six-vertex model are available for all~$q \in [1,4]$, 
and one may hope   to extend the results corresponding to Theorem~\ref{thm:main} to this range. 
Asymptotics for the one and two-arm probabilities are obtained for all~$q \in [1,4]$ in \cite{alpha1} and \cite{alpha2}, respectively, 
while that for~$\Delta$  is currently  only extended to an interval~$q \in (4-\eps,4]$ \cite{qcloseto4}.

The case of~$q=4$ is the most approachable due to the specific form of the BKW correspondence, 
but also because~$\pi_2$ and~$\Delta$ have the same asymptotics --- see Propositions~\ref{prop:1} and~\ref{prop:2} for specific statements. 
For~$1 \leq q <4$, it is expected that~$\Delta$ has a strictly faster algebraic decay than~$\pi_2$, which complicates its identification with a GFF observable.

\subsection{Main results for the four-state Potts model}

The random-cluster model with cluster-weight~$q=4$ is related, via the Edwards--Sokal coupling, to the four-state Potts model (see e.g.~\cite{duminil2017lectures} for an exposition). From the previous result, we deduce directly the following theorem for the Potts model.  

Let~$G=(V,E)$ be a finite subgraph of~$\mathbb Z^2$. A four-state Potts configuration on~$G$ is an element~$\sigma\in\{1,2,3,4\}^V$.  
The four-state Potts model on~$G$, at inverse temperature~$\beta>0$ and external field~$h\in\mathbb R$, with free boundary conditions, is the probability measure on the set of Potts configurations given by 
\[
\mathbb P_{G,\beta,h}[\sigma]
   :=\frac{1}{Z_{\rm Potts}(G,\beta,h)} 
   \exp\bigl[-\beta H_{G,h}(\sigma)\bigr],
\]
where~$Z_{\rm Potts}(G,\beta,h)$ is a normalising constant and the Hamiltonian~$H_{G,h}$ is defined as
\[
H_{G,h}(\sigma):=\sum_{e=(x,y)\in E}\mathbf 1_{\{\sigma_x\ne\sigma_y\}}
   - h\sum_{x\in V}\mathbf 1_{\{\sigma_x=1\}}.
\]

Passing to the thermodynamic limit yields infinite-volume Gibbs measures~$\mathbb P_{\mathbb Z^2,\beta,h}$; these are unique for all~$h \neq 0$. We define the magnetization
\[
m(\beta,h):=\mathbb P_{\mathbb Z^2,\beta,h}[\sigma_0=1]-\tfrac14,
\]
and the spontaneous magnetization
\[
m^*(\beta):=\lim_{h\searrow 0} m(\beta,h).
\]
The four-state Potts model undergoes a phase transition at the critical inverse temperature~$\beta_c$, 
separating the disordered phase ($m^*(\beta)=0$ for $\beta \leq \beta_c$) from the ordered phase ($m^*(\beta)>0$ for $\beta > \beta_c$). 
As in the~$q= 4$ random-cluster model, the measure with no external field is unique when~$\beta \leq \beta_c$. 
Introduce the susceptibility as 
\[
\chi(\beta):=\sum_{x\in\mathbb Z^2}\Big(\lim_{h\searrow 0}  \mathbb P_{\mathbb Z^2,\beta,h}[\sigma_0=\sigma_x]-\tfrac14\Big),
\]
and the correlation length
\[
\xi(\beta):= \Big[\lim_{n\to\infty}\lim_{h\searrow 0} -\tfrac{1}{n} \log \Big(  \mathbb P_{\mathbb Z^2,\beta,h}[\sigma_0=\sigma_{(n,0)}]-\tfrac14\Big)\Big]^{-1}.
\]
Additionally, define the free energy of the Potts model as 
\begin{align}
f(\beta) :=\lim_{n\to\infty}-\tfrac1{|\Lambda_n|}\log Z_{\rm Potts}(\Lambda_n,\beta,0).
\end{align}
The following theorem describes the precise critical behaviour of these thermodynamic quantities as~$\beta$ approaches~$\beta_c$.  

\begin{theorem}\label{cor:near critical_Potts}
For the four-state Potts model on~$\mathbb Z^2$:
\mathtoolsset{showonlyrefs=false}%
\begin{align}
\label{eq:Potts_1}f''(\beta)&=|\beta-\beta_c|^{-2/3+o(1)} & &\text{as~$\beta\to\beta_c$},\\
\label{eq:Potts_2}m^*(\beta)&=(\beta-\beta_c)^{1/12+o(1)} & &\text{as~$\beta\searrow\beta_c$},\\
\label{eq:Potts_3}\chi(\beta)&=(\beta_c-\beta)^{-7/6+o(1)} & &\text{as~$\beta\nearrow\beta_c$},\\
\label{eq:Potts_4}\mathbb P_{\mathbb Z^2,\beta_c,0}[\sigma_0=\sigma_x]-\tfrac14
   &=|x|^{-1/4+o(1)} & &\text{as~$|x|\to\infty$},\\
\label{eq:Potts_5}\xi(\beta)&=(\beta_c-\beta)^{-2/3+o(1)} & &\text{as~$\beta\nearrow\beta_c$},\\
\label{eq:Potts_6}m(\beta_c,h)&=h^{1/15+o(1)} & &\text{as~$h\to0$}.
\end{align}
\mathtoolsset{showonlyrefs=true}%
\end{theorem}

Theorem~\ref{cor:near critical_Potts} follows from Theorem~\ref{cor:near critical} and Proposition~\ref{prop:limit_two_pt} via the Edwards--Sokal coupling~\cite[Section~1.4]{grimmett2006randomcluster}.
Let us simply mention that~\eqref{eq:Potts_6} follows from~\eqref{eq:RCM_5} using a  construction involving the so-called ghost vertex, which we do not introduce here.
A complete derivation can be found in~\cite[Section~8.3]{duminil2022planar}.


\subsection*{Structure of the paper}
As mentioned before, Theorem~\ref{cor:near critical} follows directly from Theorem~\ref{thm:main} using the results of~\cite{duminil2022planar}. The rest of the paper is dedicated to Theorem~\ref{thm:main} and Proposition~\ref{prop:limit_two_pt}. 

Section~\ref{sec:background} contains a short review of the results on the random-cluster model on~$\bbZ^2$ needed in the rest of the paper. 

In Section~\ref{sec:M-formula} we link certain observables of the random-cluster model on~$\delta\bbZ^2$ to the GFF scaling limit of the corresponding six-vertex model through what we term the M-formula; see also \cite{magicformula}.
The useful observables~$\rcP_{\mathbb Z^2,p_c}[A_F(\calL)]$ are defined in terms of functions of the form 
\begin{equation}\label{eq:GFF_observables}
F=\sum_{i=1}^k \frac{q_i}{2\eps^2}\one_{B_\eps(x_i)}
\end{equation} 
for~$\eps > 0$, distinct points~$x_1\dots, x_k \in \bbR^2$ at mutual distances much larger than~$\eps$, and~$q_1, \dots,q_k \in \{\pm 1\}$. 
They are a product of contributions of the loops in the loop representation of the random-cluster model, with contributions from both small loops (i.e. loops of diameter at most~$\eps$) 
and large loops (i.e. loops surrounding different sets of balls~$B_\eps(x_i)$). Factorizing these contributions (Proposition~\ref{prop:forget small loops}) is an essential step. 

Using different functions~$F$, we then construct observables~$\widetilde\pi_1$,~$\widetilde\pi_2$ and~$\widetilde\Delta$, 
that will be shown to have the same asymptotics as~$\pi_1$,~$\pi_2$ and~$\Delta$, respectively. 
The asymptotics of these observables are derived from their GFF expressions in Sections~\ref{sec:one_arm},~\ref{sec:two_arms} and~\ref{sec:Delta}, respectively.

\subsection*{Acknowledgements.}
The authors are grateful to Emile Averous for useful discussions. 
This project has received funding from the Swiss National Science Foundation and the NCCR SwissMAP.
HDC acknowledges the support from the Simons collaboration on localization of waves.
HBC and JX are grateful to the Institut des Hautes \'Etudes Scientifiques for its hospitality and support during their postdoctoral appointments. The main part of this work was completed while HBC and JX were at IHES.
HBC acknowledges funding from the NYU Shanghai Start-Up Fund and support from the NYU–ECNU Institute of Mathematical Sciences at NYU Shanghai.

\section{Preliminaries}\label{sec:background}

This section contains a brief introduction to the two-dimensional random-cluster model. For more details and complete proofs, we direct the reader to \cite{grimmett2006randomcluster, duminil2017lectures}.

\subsection{Notation}\label{sec:notation}
Henceforth, we use the shorthand notation~$\phi_\delta=\phi_{\delta \Z^2,p_c,4}$, where~$p_c=p_c(4)=2/3$, and we set~$\phi=\phi_1$. 
Note that~$\phi_\delta$ is defined on~$\delta\mathbb Z^2$ and~$\phi$ on~$\mathbb Z^2$. 
For a finite subgraph~$G$ of~$\delta\Z^2$ and boundary condition~$\xi$, we write~$\rcP^\xi_G = \rcP^\xi_{G,p_c,4}$.

For~$x\in\R^2$ and~$0<r<R$, we write~$B_r(x):=\Ll\{z\in\R^2:\:|z-x|_2\leq r\Rr\}$. 
Also recall that~$\Lambda_r:=[-r,r]^2$ and set~$\Ann(r,R):= \Lambda_R \setminus\Lambda_r$.
Let~$\Lambda_r(x)$ and~$\Ann_x(r,R)$ denote the translates of the above by~$x \in \bbR^2$. 
Finally, for~$\delta>0$, define~$\Lambda_r^\delta := \Lambda_r\cap \delta\Z^2$.
The use of Euclidean balls~$B_r(x)$ is particularly convenient for GFF computations while 
the~$L^\infty$ balls~$\Lambda_r(x)$ are commonly used in percolation arguments. 
We choose to use them both here so as to harmonise our proofs with the pre-existing literature, in particular with~\cite{duminil2022planar}.

For two families~$(f_i)_{i\in I}$ and~$(g_i)_{i\in I}$ of positive reals, write~$f\asymp g$ (resp.~$f\les g$ and~$f\ges g$ ) if there exist constants~$c,C\in(0,\infty)$ such that, for every~$i\in I$, we have~$cg_i\leq f_i\leq Cg_i$ (resp.~$f_i\leq C g_i$ and~$f_i\geq cg_i$). In most cases, the family~$I$ will be obvious from context and omitted. 

\subsection{Qualitative properties of the critical phase}

For a rectangle~$[0,n] \times [0,k]$ with~$n,k \geq1$, let~$\calH(n,k)$ be the event that the rectangle is crossed horizontally, i.e. that its left side~$\{0\} \times [0,k]$ is connected by a path of open edges contained in the rectangle to its right side~$\{n\}\times[0,k]$. We will use the following result.

\begin{proposition}[RSW \cite{beffara2012self,DumSidTas13}]\label{prop:RSW}
    For every~$\rho > 0$, there exists~$c = c(\rho) > 0$ such that for every~$n\geq 1$ and every boundary condition~$\xi$,
    \begin{equation}\label{eq:RSW} 
    \rcP_{\Lambda_{(\rho+1)n}}^\xi[\calH(\rho n,n)]\geq c.\tag{RSW}
    \end{equation}
\end{proposition}

 For~$1 \leq r \leq R$, define the scale-to-scale versions of~$\pi_1$,~$\pi_2$ and~$\Delta$ as 
\begin{align}
    \pi_1(r,R) &:= \rcP_{}\big[\Lambda_r\longleftrightarrow \partial\Lambda_R\big],\\
        \pi_2(r,R) &:= \rcP_{}\big[\Lambda_r\longleftrightarrow \partial\Lambda_R\text{ and } \Lambda_r \stackrel{*}{\longleftrightarrow}\Lambda_R^c\big],\\
    \Delta(r,R)&:=\rcP_{\Lambda_R}^1\big[\calH(r,r)\big]-\rcP_{\Lambda_R}^0\big[\calH(r,r)\big]. \label{eq:Delta_def}
\end{align}
Also write~$\pi_1^\delta ( r,R) := \pi_1 ( r/\delta,R/\delta)$ for the analog of~$\pi_1$ on the rescaled lattice~$\delta\bbZ^2$. 
The same notation applies to~$\pi_2$ and~$\Delta$. 

These quantities obey a quasi-multiplicativity property stated below. 

\begin{proposition}[Quasi-multiplicativity]\label{p.pi_quasi-multi}
For every~$1\leq r\leq \rho\leq R$, 
\begin{align}
    \pi_1(r,R)&\asymp \pi_1(r,\rho)\pi_1(\rho,R), \label{eq:quasi_pi1}\\
     \pi_2(r,R)&\asymp \pi_2(r,\rho)\pi_2(\rho,R), \label{eq:quasi_pi2}\\
 \Delta(r,R)&\asymp  \Delta(r,\rho)\Delta(\rho,R). \label{eq:quasi_delta}
\end{align}
Moreover,~$\pi_1(R) \asymp \pi_1(1,R)$,~$\pi_2(R) \asymp \pi_2(1,R)$ and~$\Delta(R) \asymp \Delta(1,R)$.

Finally, there exists $c > 0$ such that, for any $1\leq r \leq R$, $\pi_2(r,R)  \leq \pi_1(r,R)^2\leq \pi_1(r,R)\big(\frac{r}R\big)^c$.
\end{proposition}
The quasi-multiplicativity of~$\pi_1$ and~$\pi_2$ is given by~\cite[Proposition~2.3]{duminil2022planar}; that of~$\Delta$ is given by~\cite[Theorem~1.6(ii)]{duminil2022planar}. The last property is a consequence of the FKG inequality and \eqref{eq:RSW}. 

The following mixing property can be found in \cite{duminil2022planar}.

\begin{proposition}[Mixing property]\label{prop:mixing}
There exist~$c_{\rm mix},C_{\rm mix}\in(0,\infty)$ such that for every~$r\le R/2$ and every two events~$A$ and~$B$ depending on edges in~$\Lambda_r$ and outside~$\Lambda_R$, respectively, 
\begin{equation}\label{eq:mixing}
 \big|\rcP[A\cap B]-\rcP[A]\rcP[B] \big|\le C_{\rm mix}(\tfrac rR)^{c_{\rm mix}} \rcP[A]\rcP[B].
\end{equation}
\end{proposition}

\subsection{Loop representation}

The medial lattice of $\bbZ^2$ is the graph with vertices placed at the centres of the edges of $\bbZ^2$ 
and two vertices being connected by an edge if they are at distance $1/\sqrt 2$ of each other. 
Equivalently, the vertices of the medial lattice are the intersection points between primal and dual edges,
and its edges are segments that do not intersect the primal or dual lattices. 


Fix a percolation configuration~$\omega$ on $\bbZ^2$ and recall its dual configuration $\omega^*$. The loop representation $\lpcfg(\omega)$ of $\omega$ is a partition of the edges of the medial lattice into loops and bi-infinite paths in such a way that no loop or path intersects edges of $\omega$ or $\omega^*$ --- see Figure~\ref{fig:loops} for an example. 
When either $\omega$ or $\omega^*$ contains no infinite cluster (as is the case here a.s.), $\lpcfg(\omega)$ is formed entirely of finite loops. 

\begin{figure}
\begin{center}
\includegraphics[width = 0.46\textwidth, page = 2]{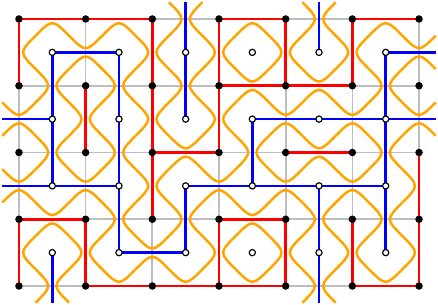}\quad
\includegraphics[width = 0.46\textwidth, page = 3]{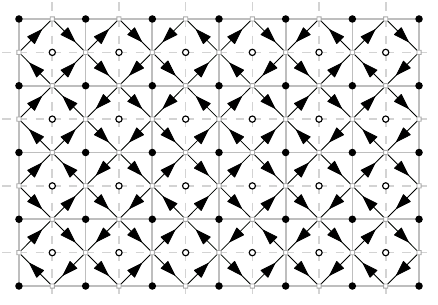}
\caption{{\em Left:} A percolation configuration in red and its dual in blue. The associated loop representation (in orange) is the set of interfaces between the primal and dual clusters. 
{\em Right:} A six-vertex configuration on the medial lattice is obtained by orienting  each loop of the loop representation.
The resulting height function is constant on the faces of the medial lattice; the orange loops are its level lines.}
\label{fig:loops}
\end{center}
\end{figure}

The loop representation plays a special role in the BKW correspondence.
When $q=4$, the correspondence is particularly simple: associate to $\omega$ the six-vertex configurations obtained by orienting uniformly the loops of $\calL(\omega)$ in either clockwise or counter-clockwise direction. See the right diagram of Figure~\ref{fig:loops} for an example
and~\cite{DumManTas16b} for more details. 

We extend the notation $\calL = \calL(\omega)$ to percolation configurations on the rescaled lattice $\delta \bbZ^2$. 
For a set of points~$X=\{x_1,\dots,x_k\}\subset \R^2$ and~$\eps>0$, set
\begin{align}\label{e.L_X,eps=}
	\calL_{X,\eps} &:= \big\{\text{$\lo\in\calL$ that do not intersect any of the balls~$B_\eps(x_i)$ for~$1\le i\le k$}\big\},\\
	\calL_{X,\eps}^{\rm odd}&:=\big\{\lo\in\calL_{X,\eps}\text{ surrounding an odd number of  balls~$B_\eps(x_i)$ for~$1\le i\le k$}\big\}. \label{e.L^odd_X,eps=}
\end{align}
These sets of loops will be useful when relating GFF observables as in \eqref{eq:GFF_observables} to events in the random-cluster model.

\subsection{Mixing conditioned on one-arm events}
Using the notation above for $X=\{0\}$, define the measure 
\begin{equation}
\phi_\delta\big[\cdot\,\big|\,\calL_{\{0\},\eps}^{\rm odd}=\emptyset\big] = \lim_{R\to\infty} \phi_{\Lambda^\delta_R}^1\big[\cdot\,\big|\,\calL_{\{0\},\eps}^{\rm odd}=\emptyset\big].
\end{equation}
This is a coarse-grained version of the so-called {\em Incipient Infinite Cluster} (IIC) in which there exists either an open or a closed path from~$ B_\eps(0)$ to infinity. The existence of the limit is justified in \cite{basu2017kesten, garban2018scaling, kesten1986incipient, jarai2003invasion}, as is the following quantitative mixing property. 

\begin{lemma}\label{lem:mixing_IIC}
There exist constants~$C,c>0$ such that for every~$R/2 \ge r  \ge \ep \ge \delta>0$, every event~$A$ depending on~$\Lambda_r^\delta$ and every boundary condition~$\xi$ on~$\Lambda_R^\delta$, 
\begin{equation*}
\big| \phi_{\Lambda_R^\delta}^\xi[A\,|\,\calL_{\{0\},\eps}^{\rm odd}=\emptyset] - \phi_\delta[A\,|\,\calL_{\{0\},\eps}^{\rm odd}=\emptyset] \big|\le C\big(\tfrac{r}{R}\big)^c.
\end{equation*}
\end{lemma}

The following property states that the restriction of~$\phi_\delta[\cdot\,|\,\calL_{\{0\},\eps}^{\rm odd}=\emptyset]~$ to~$B(2\ep)$ is absolutely continuous with respect to~$\phi_\delta$, with a uniformly bounded Radon-Nikodym derivative. 

\begin{lemma}\label{lem:mixing_IIC0}
There exists~$C>0$ such that for every~$\ep \ge \delta>0$ and every event~$A$ depending on~$B_{2\eps}(0)$,  
\begin{equation*}
\phi_\delta\big[A\,\big|\,\calL_{\{0\},\eps}^{\rm odd}=\emptyset\big]\le C\phi_\delta\Ll[A\Rr].
\end{equation*}
\end{lemma}

We include the brief proof of this statement.

\begin{proof} 
Fix~$\delta$, $\eps$ and~$A$ as in the statement. All constants below are independent of these choices. The inclusion~$\{\calL_{\{0\},\eps}^{\rm odd}=\emptyset\}\subset \{\calL_{\{0\},4\eps}^{\rm odd}=\emptyset\}$ implies that 
\begin{equation}
	\phi_\delta[A\,|\,\calL_{\{0\},\eps}^{\rm odd}=\emptyset]
	=\lim_{N\rightarrow\infty}\frac{\phi_{\Lambda^\delta_N}^1[A,\calL_{\{0\},\eps}^{\rm odd}=\emptyset]}{\phi_{\Lambda^\delta_N}^1[\calL_{\{0\},\eps}^{\rm odd}=\emptyset]}
	\le\lim_{N\rightarrow\infty}\frac{\phi_{\Lambda^\delta_N}^1[A,\calL_{\{0\},4\eps}^{\rm odd}=\emptyset]}{\phi_{\Lambda^\delta_N}^1[\calL_{\{0\},\eps}^{\rm odd}=\emptyset]}.
\end{equation}
The event~$\{\calL_{\{0\},4\eps}^{\rm odd}=\emptyset\}$ only depends on the edges in~$B(4\ep)^c$. 
As such,  the mixing property~\eqref{eq:mixing} implies that 
\begin{equation}
\phi_{\Lambda^\delta_N}^1[A,\calL_{\{0\},4\eps}^{\rm odd}=\emptyset] \leq C_{\rm mix}\phi_{\Lambda^\delta_N}^1[A]\phi_{\Lambda^\delta_N}^1[\calL_{\{0\},4\eps}^{\rm odd}=\emptyset].
\end{equation}
Moreover, a direct application of the Russo--Seymour--Welsh theorem~\eqref{eq:RSW}, gives the existence of an absolute constant~$C_1 > 0$ such that 
\begin{equation}
\phi_{\Lambda^\delta_N}^1[\calL_{\{0\},4\eps}^{\rm odd}=\emptyset] \leq C_1 \phi_{\Lambda^\delta_N}^1[\calL_{\{0\},\eps}^{\rm odd}=\emptyset].
\end{equation}
We conclude that 
\begin{equation}
\phi_\delta[A\,|\,\calL_{\{0\},\eps}^{\rm odd}=\emptyset]
\le C_1C_{\rm mix}\lim_{N\rightarrow\infty}\phi_{\Lambda^\delta_N}^1[A] 
=C_1C_{\rm mix}\phi_\delta[A],
\end{equation}
as required. 
\end{proof}

\subsection{Tightness for the number of macroscopic loops}
For~$\eps,\eta,\lambda>0$, set 
\begin{equation*}
E(\eps,\eta,\lambda):=\{\text{there exist more than~$\lambda/\eta^2$ loops of diameter at least~$\eta \eps/2$ intersecting~$B_\eps(0)$}\}.
\end{equation*}

\begin{lemma}[Contribution of loops of given diameter]\label{lem:small loops}
There exist~$C,c\in(0,\infty)$ such that for every~$\lambda>0$,~$0<\eta<1$, and~$0<\delta<\eps$,
\begin{equation*}
\phi_\delta[E(\eps,\eta,\lambda)]\le C\exp(-c\lambda).
\end{equation*}
\end{lemma}

\begin{proof}
It suffices to consider the case~$\eps=1$ since the general case follows by substituting~$(\eps, \eta\eps)$ for~$(1, \eta)$.

Let~$N_0$ be half the number of crossings from~$\Lambda_{\eta/8}$ to~$\Lambda_{\eta/4}^c$ in the union of loops in~$\calL$  (each loop crosses the annulus an even number of times, so~$N_0$ is an integer).
Having~$N_0 \geq k$ implies the existence of~$k$ distinct clusters in~${\rm Ann}(\tfrac18\eta, \tfrac14\eta)$ crossing the annulus from its inside to outside boundary. 

Then the~\eqref{eq:RSW} property implies the existence of a constant~$u>0$ independent of~$\eta$ and~$\delta$ such that, 
conditioned on the configuration in~$\Lambda_{\eta/4}^c$,~$N_0$ is stochastically dominated by a geometric random variable of parameter~$u>0$. 


For~$x\in B_1(0)\cap \tfrac14\eta\mathbb Z^2$, let~$N_x$ be obtained by translating~$N_0$ by~$x$. 
That is,~$N_x$ is half the number of crossings of~${\rm Ann}_x(\tfrac18\eta, \tfrac14\eta)$ in~$\calL$. 
Then, any loop of diameter at least~$\eta/2$ contributes to at least one variable~$N_x$. 
As such, 
\begin{equation}
E(1,\eta,\lambda) \subset \big\{\sum_{x }N_x\ge \lambda /\eta^{2}\big\}.
\end{equation}

We may partition~$B_1(0)\cap \tfrac14\eta\mathbb Z^2$ into 9 sets~$I_1,\dots,I_9$ such that, for each~$i = 1,\dots, 9$, the boxes~$\Lambda_{\eta/4}(x)$ with~$x \in I_i$ are disjoint (except possibly on their boundaries). Moreover, each set~$I_i$ contains at most~$16\eta^{-2}$ points. 
It follows that, for each~$i$, the variables 
$N_x$ for~$x\in I_i$ may be stochastically dominated by independent geometric random variables~$\mathbf N_x$ of parameter~$u>0$.

We conclude that
\begin{equation*}
\phi_\delta[E(1,\eta,\lambda)]\le \sum_{i=1}^9\phi_\delta\Big[\sum_{x\in I_i}N_x\ge \lambda/(9\eta^2)\Big]\le  \sum_{i=1}^9\mathbb P \Big[\sum_{x\in I_i}\mathbf N_x\ge \lambda/(9\eta^2)\Big].
\end{equation*}
The result follows from large deviation theory applied to sums of independent geometric variables.
 \end{proof}

\subsection{A technical lemma for quasi-multiplicative sequences} 
The following easy lemma will be used repeatedly.

\begin{lemma}\label{l.exp_quasi_multi}
Let~$\upsilon:[0,\infty)^2\rightarrow (0,\infty)$ be a function that is increasing in the first variable, decreasing in the second and that 
 satisfies the quasi-multiplicativity relation
\begin{equation*}
    \upsilon(r,\rho)\upsilon(\rho,R) \asymp \upsilon(r,R),\quad\text{for every }0< r\leq \rho\leq R.
\end{equation*}
If there is some constant~$\varsigma>0$ such that
\begin{align}\label{e.limlimlogupsilon/log=alpha}
    \lim_{\eps\to0}\liminf_{N\to\infty} \frac{\log\upsilon(\eps N,N)}{\log \eps} = \lim_{\eps\to0}\limsup_{N\to\infty} \frac{\log\upsilon(\eps N,N)}{\log \eps} = \varsigma,
\end{align}   
then for any positive sequences~$(r_n)_{n\in\N}$ and~$(R_n)_{n\in\N}$ with~$\frac{R_n}{r_n} \to \infty$ as~$n \to \infty$, 
\begin{align}\label{e.limlogupsilon/log=alpha}
    \lim_{n\to\infty}\frac{\log\upsilon(r_n,R_n)}{\log(R_n/r_n)} =-\varsigma.
\end{align}
In particular,
\begin{equation}\label{e.limlogupsilon/log(1,N)=alpha}
\lim_{N\to\infty}\frac{\log\upsilon(1,N)}{\log N}=-\varsigma.
\end{equation}
\end{lemma}

\begin{proof}
We start with the upper bound in~\eqref{e.limlogupsilon/log=alpha}.
First, we fix any~$0<\eps<1$. We write~$K_n:= \left\lfloor\frac{\log R_n}{\log\eps^{-1}}\right\rfloor-1$ and~$k_n :=\left\lceil\frac{\log r_n}{\log\eps^{-1}}\right\rceil$. The quasi-multiplicativity and monotonicity of~$\upsilon$ implies that there exists a constant~$C>0$ such that
\begin{align*}
    \upsilon(r_n,R_n) \leq C^{K_n-k_n}\prod^{K_n}_{k=k_n}\upsilon (\eps^{1-k},\eps^{-k}).
\end{align*}
Taking the logarithm and dividing both sides by~$\log(R_n/r_n)\geq(K_n-k_n)\log\eps^{-1}$, we get
\begin{align*}
    \frac{\log\upsilon(r_n,R_n)}{\log (R_n/r_n)}\leq \frac{1}{\log \eps^{-1}}\Ll(\log C + \frac{1}{K_n-k_n}\sum_{k=k_n}^{K_n} \log(\upsilon (\eps^{1-k},\eps^{-k})) \Rr).
\end{align*}
Since~$\frac{R_n}{r_n}$ tends to~$\infty$, we have~$K_n - k_n$ tends to~$\infty$. Thus, 
\begin{align*}
    \limsup_{n\to\infty}\frac{1}{K_n-k_n}\sum_{k=k_n}^{K_n}  \log(\upsilon (\eps^{1-k},\eps^{-k}))\leq \limsup_{k\to\infty}  \log(\upsilon (\eps^{1-k},\eps^{-k})) \leq\limsup_{N\to\infty} \log\upsilon(\eps N,N),
\end{align*}
and therefore
\begin{align*}
    \limsup_{n\to\infty} \frac{\log\upsilon(r_n,R_n)}{\log (R_n/r_n)}\leq \frac{\log C}{\log \eps^{-1}} - \liminf_{N\to\infty} \frac{\log\upsilon(\eps N,N)}{\log\eps}.
\end{align*}
Now, sending~$\eps$ to~$0$ and using~\eqref{e.limlimlogupsilon/log=alpha}, we obtain the upper bound matching~\eqref{e.limlogupsilon/log=alpha}. 

The lower bound follows analogously. Substituting~$1$ and~$N$ for~$r_n$ and~$R_n$ in~\eqref{e.limlogupsilon/log=alpha}, we find the particular result.
\end{proof}

\section{The M-formula and its implications}\label{sec:M-formula}

\subsection{The M-formula}

The main ingredient of our argument is Theorem~\ref{thm:Monday's formula} below, which combines the Baxter-Kelland-Wu coupling \cite{baxter1976equivalence} with the convergence of the six-vertex model to the Gaussian Free Field \cite{DumKozLamMan26}. A more complex formula~\cite{magicformula} is available for random-cluster models with~$1 \leq q \leq 4$. The simpler form below is due to the particular weights appearing in the Baxter-Kelland-Wu coupling for~$q = 4$. 


A collection \(L\) of non-self-crossing loops is said to be \textit{locally finite} if, for every compact \(D\subset \R^2\), there are only finitely many loops in \(L\) that intersect \(D\) nontrivially. For every integrable, compactly supported function \(F:\R^2\to\R\) with zero mean (i.e.\ \(\int_{\R^2}F(z)\,\d z=0\)), and every locally finite collection \(L\) of non-self-crossing loops, we define
\begin{equation}\label{e.A_F(L)=}
\mathcal A_F(L):=\prod_{\ell\in L}\cos\Ll(\int_{\itr(\ell)} F(z)\,\d z\Rr)
\end{equation}
where $\itr(\lo)$ is the complement of the unbounded component of $\R^2\setminus\ell$.
By the mean-zero condition, the only loops \(\ell\) for which $\cos\Ll(\int_{\itr(\ell)} F(z)\,\d z\Rr)\neq1$ must intersect a fixed compact set containing \(\supp F\). There are only finitely many such loops in $L$, hence \(\mathcal A_F(L)\) is well-defined even when \(L\) contains infinitely many loops.

Since the whole-plane two-dimensional GFF is defined modulo additive constants as a random distribution in the local Sobolev space \(W^{-s,2}(\R^2)\) for every \(s>0\), it is naturally tested against compactly supported functions in the dual space \(W^{s,2}(\R^2)\). The quotient by additive constants imposes the mean-zero condition on the test functions.

\begin{theorem}[M-formula]\label{thm:Monday's formula}
    For every mean-zero and compactly supported $F\in W^{s,2}(\R^2)$ for some~$s\in(0,1)$, we have
    \begin{equation}\label{eq:Monday's formula}
        \lim_{\delta\rightarrow 0} \rcE_{\delta}\Ll[\mathcal A_F(\lpcfg)\Rr]=\exp\Big( \frac{1}{2\pi^2} \iint_{\R^2\times\R^2} F(z)F(z')\log|z-z'| \d z\d z' \Big).
    \end{equation}
\end{theorem}


\begin{proof}[Proof of Theorem~\ref{thm:Monday's formula}]
Fix~$F$ as in the statement. 
The Baxter--Kelland--Wu correspondence (see \cite{baxter1976equivalence,DumManTas16b,magicformula}) implies that 
\begin{align*}
    \rcP_{\delta}[\mathcal A_F(\calL)]=\mathbb E^{\rm 6V}_{\delta}\Big[\exp\Big(i\int_{\R^2} F(z)h(z)\d z\Big)\Big],
\end{align*}
where~$\mathbb E^{\rm 6V}_{\delta}$ is the full-plane six-vertex measure on the medial lattice of~$\delta \Z^2$ with weights $\mathbf{a}=\mathbf{b}=1$ and $\mathbf{c}=2$ and~$h$ is the associated gradient height function extended in a constant fashion on each face. We refer to~\cite[Sec. 2.2]{DumKozLamMan26} for details.

Fix an open bounded set $U\subset \R^2$ that contains $\supp F$, so that $F\in W^{s,2}(U)$.
By the recent result~\cite[Theorem~2.8]{DumKozLamMan26}\footnote{See~\cite[Def.\ 2.7 (iii)]{DumKozLamMan26} for the notion of convergence used here.}, the six-vertex height function (at the point corresponding to~$q=4$ random-cluster model) converges to the Gaussian Free Field with variance~$2/\pi$ in the topology of $W^{-s,2}(U)$, whose law is denoted by~$\mathbb P^{\rm GFF}$.
In particular, the random variable~$\int_{\R^2} F(z)h(z)\d z$ under the measures~$\mathbb P^{\rm 6V}_{\delta}$ converges in law as~$\delta$ tends to $0$
 to a centered normal variable of variance~$ \frac{1}{\pi^2} \iint F(z)F(z')\log|z-z'| \d z\d z'$.
This implies that the characteristic function of the former converges to that of the latter, which is to say that 
\begin{align*}
	\lim_{\delta\rightarrow0}\mathbb E^{\rm 6V}_{\delta}\Big[\exp\Big(i\int F(z)h(z)\d z\Big)\Big]
	&=\mathbb E^{\rm GFF}\Big[\exp\Big(i\int F(z)h(z)\d z\Big)\Big]\\
	&=\exp\Big(\frac{1}{2\pi^2}\iint F(z)F(z')\log|z-z'| \d z\d z' \Big).
\end{align*}
Combining the previous two displayed equations concludes the proof.
\end{proof}


\subsection{The normalization factor}

The following quantity will play an essential role in the statements below. 

\begin{definition}\label{def:kappa}
For $0<\delta<\eps$, define the {\em normalization factor} as
 \begin{equation*}
    \bc^\delta(\eps):=\phi_\delta\big[\mathcal A_{\frac1{2\eps^2}\one_{B_\eps(0)}}(\calL)\,\big|\,\calL_{\{0\},\eps}^{\rm odd}=\emptyset\big].
    \end{equation*}
    \end{definition}

In the above, the function~$\frac1{2\eps^2}\one_{B_\eps(0)}$ is not of zero average, and~$\mathcal A_{\frac1{2\eps^2}\one_{B_\eps(0)}}(\calL)$ is a priori ill-defined\footnote{Actually, due to the choice of function, any loop surrounding~$B_\eps(0)$ contributes a factor~$0$ to~$\mathcal A_{\frac1{2\eps^2}\one_{B_\eps(0)}}$ so~$\mathcal A_{\frac1{2\eps^2}\one_{B_\eps(0)}}(\calL) = 0$~$\phi_\delta$-a.s..} under~$\phi_\delta$ due to the existence of infinitely many loops surrounding~$B_\eps(0)$. Note however that under 
$\phi_\delta[\cdot\,|\,\calL_{\{0\},\eps}^{\rm odd}=\emptyset]$, no loops surround~$B_\eps(0)$ and~$\mathcal A_{\frac1{2\eps^2}\one_{B_\eps(0)}}(\calL)$ is a.s.\ a finite product.

\begin{lemma}\label{l.kappa}
    There exists~$c>0$ such that for every~$0<\delta<c\eps$,
    \begin{equation*}
    c\le \bc^\delta(\eps)\le 1.
    \end{equation*}
\end{lemma}

\begin{proof}
Set~$F:=\frac1{2\eps^2}\one_{B_\eps(0)}$. 
Observe that~$\mathcal A_{F}(\calL) \in [0,1]$ as for any loop~$\ell$, we have~$\int_{\itr(\lo)} F(z)\d z \in [0,\pi/2]$, which means that its contribution $\cos\Ll(\int_{\itr(\lo)} F(z)\d z\Rr)$ to~$\mathcal A_{F}(\calL)$ is between~$0$ and~$1$. 
The upper bound on~$\bc^\delta(\eps)$ follows. 

We turn to the lower bound. 
Let~$E$ be the event that there exists an open circuit in~$B_{2\eps}(0)$ surrounding~$B_\eps(0)$. 
Then, by the inclusion of events, for any~$\eta > 0$, 
\begin{align}
	\bc^\delta(\eps)
	&\ge \eta\phi_\delta\Ll[E,\mathcal A_F(\lpcfg)> \eta\,|\,\calL_{\{0\},\eps}^{\rm odd}=\emptyset\Rr]\\
	&\ge \eta\Ll(\phi_\delta\Ll[E\,|\,\calL_{\{0\},\eps}^{\rm odd}=\emptyset\Rr]-\phi_\delta\Ll[E, \mathcal A_F(\lpcfg)\leq \eta\,|\,\calL_{\{0\},\eps}^{\rm odd}	=\emptyset\Rr]\Rr)\\
	 &\ge \eta\Ll(\phi_\delta\Ll[E\,|\,\calL_{\{0\},\eps}^{\rm odd}=\emptyset\Rr]-C \phi_\delta\Ll[E, \mathcal A_F(\lpcfg \setminus \calL_{\{0\},\eps}^{\rm odd})\leq \eta\Rr]\Rr), 
	 \label{eq:h00}
\end{align}
where~$C$ is the constant provided by Lemma~\ref{lem:mixing_IIC0}, and where we used the fact that the event~$E\cap\{ \mathcal A_F(\lpcfg\setminus \calL_{\{0\},\eps}^{\rm odd})\leq \eta\}$ depends only on loops included in~$B_{2\eps}(0)$.

Self-duality (for the first inequality) and the FKG inequality together with~\eqref{eq:RSW} (for the second) imply the existence of a constant~$c_0>0$ independent of~$\eps$ and~$\delta$ such that
\begin{equation}\label{eq:h01}
	\phi_\delta[E\,|\,\calL_{\{0\},\eps}^{\rm odd}=\emptyset] \ge
	\tfrac12 \phi_\delta[E\,|\,B_\eps(0) \lra \infty] \ge
	  c_0, 
\end{equation}
with the measure in the middle corresponding to the incipient infinite cluster measure conditioned on having an infinite primal cluster --- its definition is similar to that of~$\phi_\delta[\cdot\,|\,\calL_{\{0\},\eps}^{\rm odd}=\emptyset]~$. 

We will now prove that, for~$\eta$ sufficiently small, independent of~$\delta$ and~$\eps$, 
\begin{equation}\label{eq:h1}
\phi_\delta[\mathcal A_F(\lpcfg \setminus \calL_{\{0\},\eps}^{\rm odd})\le \eta]\le \frac{c_0}{2C}.
\end{equation}
We do so by bounding the contribution to~$\mathcal A_F(\lpcfg \setminus \calL_{\{0\},\eps}^{\rm odd})$ of loops depending on their diameter. 
Write~$\calL_{(a,b]}$ for the collection of the loops in~$\calL$ that intersect~$B_\eps(0)$ and whose diameters belong to  the interval~$(a,b]$.

Proposition~\ref{prop:RSW} easily implies the existence of~$u>0$ such that the probability that there exists a loop intersecting~$B_\eps(0)$ and surrounding an area of~$B_\eps(0)$ which is larger than~$(1-u)\pi\eps^2$ is bounded by~$c_0/(8C)$. 
Also, by Lemma~\ref{lem:small loops}, there exists~$K>0$ such that the probability that there are more than~$K$ loops of diameter at least~$\eps$ intersecting~$B_\eps(0)$ is bounded by~$c_0/(8C)$. Thus, 
\begin{equation}\label{eq:b}
\phi_\delta\Ll[\mathcal A_F\Ll(\calL_{(\eps,\infty]}\Rr)\le \cos\Ll(\tfrac\pi2(1-u)\Rr)^K\Rr]\le \frac{c_0}{4C}.
\end{equation}

Next, we consider the contribution of loops in~$\calL_{(2^{-k-1}\eps,2^{-k}\eps]}$ for~$k\geq 0$. 
For any such loop~$\ell$, 
$$ \cos \Big(\int_{{\rm int}(\ell)}F(z)\d z \Big) \geq  \cos \Big(\tfrac{(2^{-k}\eps)^2 \pi}{8\eps^2}\Big) \geq  \exp(-c_1 4^{-2k})~$$
for some constant~$c_1 >0$ independent of~$\delta$,~$\eps$  or~$k$, which is chosen such that~$\cos( x \frac\pi{8}) \geq  \exp(-c_1 x^2)$ for all~$0 \leq x \leq 1$. 
Thus, Lemma~\ref{lem:small loops} implies that, for any~$\lambda > 0$, 
\begin{align*}
\phi_\delta \big[\mathcal A_F\Ll(\mathcal  L_{(2^{-k-1}\eps,2^{-k}\eps]} \Rr)\le e^{-c_1 4^{-k} \lambda }\big]
&\le\phi_\delta\big[\Ll|\mathcal  L_{(2^{-k-1}\eps,2^{-k}\eps]}\Rr| \ge \lambda 4^{k} \big] \le C'\exp(-c\lambda)
\end{align*}
for  constants~$C',c>0$ independent of~$k$,~$\lambda$,~$\eps$ or~$\delta$.
Using the above for~$\lambda = \alpha 2^{k}$ for some~$\alpha > 0$, we conclude that 
\begin{align}
\phi_\delta\Big[\mathcal A_F\Ll(\calL_{(0,\eps]}\Rr)\leq \prod_{k\ge0}e^{-c_1\alpha 2^{-k}}\Big]
&\le \sum_{k=0}^\infty\phi_\delta\Big[\mathcal A_F\Ll(\calL_{(2^{-k-1}\eps,2^{-k}\eps]}\Rr)\le e^{-c_1\alpha 2^{-k}} \Big]\notag
\\
&\le \sum_{k\ge0}C'\exp\Ll(-c\alpha 2^k\Rr)\le \frac{c_0}{4C}\label{eq:bb}
\end{align}
where the last inequality may be ensured by taking~$\alpha$ sufficiently small, independent of~$\eps$ and~$\delta$. 
Combining~\eqref{eq:b} and~\eqref{eq:bb}, we obtain~\eqref{eq:h1} for~$\eta:= \cos(\tfrac\pi2(1-u))^K e^{-2c_1 \alpha} > 0$.

Finally,~\eqref{eq:h1} together with~\eqref{eq:h00} and~\eqref{eq:h01} imply that~$\bc^\delta(\eps) \ge \eta c_0/2.$
\end{proof}

\begin{remark}
The proof above also shows that~$\bc^\delta(\eps)$ is well approximated by the contributions of loops that are macroscopic. As a consequence, when considering a subsequential scaling limit~$(\delta_n)$,~$\bc^{\delta_n}(\eps)$ converges to a constant~$\bc(\eps)$. However, we will not use this property in this paper. 
\end{remark}

\subsection{Factorizing small loops}

For~$\eps>0$, define the set~$\calF_\eps$ of functions~$F$ of the form 
\begin{equation}\label{e.F=sum_pm_pi/2}
F=\sum_{i=1}^k \frac{q_i}{2\eps^2}\one_{B_\eps(x_i)}
\end{equation} 
with points~$x_1, \dots, x_k \in \bbR^2$ at mutual distances strictly larger than~$2\eps$,
and with~$q_i\in\{\pm1\}$ for every~$i$. Notice that we allow~$F\in \calF_\eps$ to have non-zero mean.
Since \(B_\eps(x_i)\) has Lipschitz boundary, it follows either by a direct verification using the Slobodeckij seminorm or from~\cite[Lemma~4]{sickel2020regularity}, which is formulated in terms of Besov spaces, that such \(F\) belongs to \(W^{s,2}(\R^2)\) for every \(s<\frac12\). Therefore, whenever \(F\) also has zero mean, Theorem~\ref{thm:Monday's formula} applies.

The following proposition states that, for functions in~$\mathcal F_\eps$, we can factorize the contributions of loops intersecting the balls~$B_\eps(x_i)$ and those surrounding these balls. The former are captured by the normalizing factors~$\bc^\delta(\eps)$ from Definition~\ref{def:kappa}.
    
\begin{proposition}\label{prop:forget small loops}
There exists~$c>0$ such that the following holds true. For every integer~$k\ge1$ and~$\dst\in(0,1)$, there exists a constant~$C=C(k,\dst)>0$ such that for every~$0<\delta\le \eps<\dst/2$ and every~$F\in\mathcal F_\eps$ with~$k$ points~$x_1,\dots,x_k\in B_{1/\dst}(0)$ at mutual distances at least~$\dst$ of each other,
\begin{equation}\label{e.prop:forget small loops}
\Big|\bc^\delta(\eps)^k\phi_\delta[\mathcal A_F(\calL_{\{x_1,\dots,x_k\},\eps})]-\phi_\delta[\mathcal A_F(\calL)]\Big|\le C\pi_1^\delta(\eps,\dst)^k(\tfrac\eps\dst)^c.\end{equation}
\end{proposition}

\begin{remark}\label{rem:no_odd_loops}
For~$F\in\mathcal F_\eps$, the loops in~$ \calL_{\{x_1,\dots,x_k\},\eps}$  satisfy~$\int_{{\rm int}(\ell)}F(z)\d z\in \tfrac\pi2\mathbb Z$. Furthermore, if~$\ell\in\calL_{\{x_1,\dots,x_k\},\eps}^{\rm odd}$, then~$\int_{{\rm int}(\ell)}F(z)\d z\in\tfrac{\pi}{2}+\pi\Z$. Consequently,
\begin{equation}\label{e.|A|<0,or1}
|\mathcal A_F(\calL_{\{x_1,\dots,x_k\},\eps})| \leq |\mathcal A_F(\calL_{\{x_1,\dots,x_k\},\eps}^{\rm odd})|:=\begin{cases} \ 0&\text{ if }\calL_{X,\eps}^{\rm odd}\ne\emptyset,\\
\ 1&\text{ if }\calL_{X,\eps}^{\rm odd}=\emptyset.
\end{cases}
\end{equation}
When~$\calL^\mathrm{odd}_{\{x_1,\dots,x_k\},\eps}=\emptyset$, each annulus ${\rm Ann}_{x_i}(\eps,\dst/2)$ is crossed from the inside to the outside by either a primal or dual arm. Here, a primal arm means an open path in the primal lattice contained in the annulus and connecting its inner boundary to its outer boundary, while a dual arm means an open path in the dual configuration. 
Applying a union bound, self-duality and mixing~\eqref{eq:mixing}, we conclude that 
\begin{align}\label{eq:h1b}
	\phi_\delta\big[|\mathcal A_F(\calL)|\big]  \leq \phi_\delta[\calL^\mathrm{odd}_{\{x_1,\dots,x_k\},\eps}=\emptyset]\le C^k \pi_1^\delta(\eps,\dst)^{k}
\end{align}
for some constant~$C$ independent of all quantities above.
Hence, for reasonable choices of the points~$x_1,\dots, x_k$, one should think of~$\pi_1^\delta(\eps,\dst)^{k}$ as the typical order of magnitude of~$\phi_\delta[\mathcal A_F(\calL)]$, and therefore the right-hand side of~\eqref{e.prop:forget small loops} should be perceived as a lower-order term. 
This heuristic is not entirely true, in particular because the sign of~$\mathcal A_F(\calL)$ may lead to cancellations, but it will suffice for our purposes.
\end{remark}

\begin{proof}
Fix~$d/2 >\eps\geq \delta >0$,~$X=\{x_1,\dots,x_k\}$ and~$F\in\calF_\eps$ as in the statement; all constants below are independent of these choices. 
Define~$\dst \geq R\geq r \geq \eps$ by
$\tfrac\eps r=\tfrac rR=\tfrac R\dst=\left(\tfrac \eps\dst\right)^{1/3}$.
For~$x\in\mathbb R^2$, introduce the events
\begin{align*}
E_{\rm out}&:=\big\{\forall
\ell\in \calL:\: \forall i\in\llbracket1,k\rrbracket,\,\ell\cap \partial B_R(x_i)=\emptyset\text{ or }\ell\cap \partial B_{\dst/2}(x_i)=\emptyset\big\},\\
E_x&:=\big\{\forall \ell\in \calL:\: \ell\cap B_\eps(x)=\emptyset \text{ or }\ell\cap \partial B_r(x)= \emptyset\big\}.
\end{align*}
The first event states that no loop crosses an annulus of  inner radius~$R$ and outer radius~$d/2$ around one of the~$x_i$ from the inside to the outside. The second states that no loop crosses the annulus of inner radius~$\eps$ and outer radius~$r$ around~$x$. 
The annulus in the second case is much smaller than the one in the first.


When,~$E_{\rm out}^\comple$ occurs, there exists~$i$ such that~${\rm Ann}_{x_i}(R,\dst/2)$ is crossed by an interface from inside to outside. 
Similarly, when one of the events~$E_{x_i}$ fails to occur, one of the annuli~${\rm Ann}_{x_i}(r,R)$ contains an interface crossing from inside to outside. 
By a similar reasoning as for~\eqref{eq:h1b} and using Proposition~\ref{p.pi_quasi-multi}, 
\begin{align}
	\phi_\delta \big[ \big(E_{\rm out}\cap E_{x_1}\cap \dots\cap E_{x_k}\big)^c  \cap \{ \calL_{\{0\},\eps}^\mathrm{odd}=\emptyset\} \big]
	&\le  k C^k \pi_1^\delta(\eps,\dst)^{k} \Big( \frac{\pi_2^\delta(\eps,r)}{\pi_1^\delta(\eps,r)} +  \frac{\pi_2^\delta(r,R)}{\pi_1^\delta(r,R)}\Big)   \\
	&\le  k\, C^k \pi_1^\delta(\eps,\dst)^{k} (\tfrac\eps\dst)^c,
	\label{eq:hh_22}
\end{align}
for constants~$C,c>0$.
In particular, due to~\eqref{e.|A|<0,or1}, 
\begin{align}
\Big| \phi_\delta\Ll[\mathcal A_F(\calL)\one_{E_{\rm out}\cap E_{x_1}\cap \dots\cap E_{x_k}}\Rr]  - \phi_\delta\Ll[\mathcal A_F(\calL)\Rr]  \Big| 
&\leq k\, C^k \pi_1^\delta(\eps,\dst)^{k} (\tfrac\eps\dst)^c \qquad\text{ and } \\
\Big| \phi_\delta\Ll[\mathcal A_F(\calL_{X,\eps})\one_{E_{\rm out}}\Rr]  - \phi_\delta\Ll[\mathcal A_F(\calL_{X,\eps})\Rr]  \Big| 
&\leq k\, C^k \pi_1^\delta(\eps,\dst)^{k} (\tfrac\eps\dst)^c.	\label{eq:hh_23}
\end{align}
Similarly, for the IIC measure,
\begin{align}
\phi_\delta[E_0^\comple\,|\,\calL_{\{0\},\eps}^\mathrm{odd}=\emptyset]&\le C\left(\tfrac\eps \dst\right)^{c},\label{eq:h1c}
\end{align}
with a potentially altered value of~$C$.

%

Let now~$\calL_i$ be the set of loops included in~$B_{\dst/2}(x_i)$, and~$\calL_{\rm out}$ be the loops in~$\calL$ that are in none of the~$\calL_i$ for~$1\le i\le k$. Write~$F_0:=\tfrac{1}{2\eps^2}\one_{B_\eps(0)}$. Then
\begin{align}
\label{eq:a}
\begin{split}
\phi_\delta\Ll[\mathcal A_F(\calL)\one_{E_{\rm out}\cap E_{x_1}\cap \dots\cap E_{x_k}}\Rr]&=\phi_\delta\Ll[\Ll(\prod_{i=1}^k\mathcal A_F(\calL_i)\Rr)\mathcal A_F(\calL_{\rm out})\one_{E_{\rm out}\cap E_{x_1}\cap\dots\cap E_{x_k}}\Rr]\\
&=\phi_\delta\Ll[\Ll(\prod_{i=1}^k\phi_\delta\Ll[\mathcal A_{F_0(\cdot\,-x_i)}(\calL_i)\one_{E_{x_i}}\,\big|\,\calL_{\rm out}\Rr]\Rr)\mathcal A_F(\calL_{\rm out})\one_{E_{\rm out}}\Rr],
\end{split}
\end{align}
where we used that~$E_{\rm out}$ is measurable in terms of~$\calL_{\rm out}$ and that, conditionally on~$\calL_{\rm out}$, the families of loops~$\calL_i$ are independent.

Fix  a realisation~$L_{\rm out}$ of~$\calL_{\rm out}$ that satisfies~$E_{\rm out}$.
Recall that~$\mathcal A_{F_0(\cdot\,-x_i)}(\calL_i) = 0$ if~$\calL_i \cap \calL_{X,\eps}^{\rm odd}\neq\emptyset$. 
Thus
\begin{align}
 &\phi_\delta\Ll[\mathcal A_{F_0(\cdot\,-x_i)}(\calL_i)\one_{E_{x_i}}\,\big|\,\calL_{\rm out} = L_{\rm out}\Rr] \\
&\qquad =\phi_\delta\Ll[\mathcal A_{F_0(\cdot\,-x_i)}(\calL_i)\one_{E_{x_i}}\,\big|\,\calL_i \cap \calL_{X,\eps}^{\rm odd}=\emptyset \text{ and }\calL_{\rm out} = L_{\rm out}\Rr]
\phi_\delta\Ll[\calL_i \cap \calL_{X,\eps}^{\rm odd}=\emptyset\,\big|\,\calL_{\rm out} = L_{\rm out}\Rr].
\label{eq:h1d}
\end{align}
Under the conditioning~$\calL_{\rm out} = L_{\rm out}$ and~$\calL_i \cap \calL_{X,\eps}^{\rm odd}=\emptyset$, 
the variable~$\mathcal A_{F_0(\cdot\,-x_i)}(\calL_i)\one_{E_{x_i}}$ is entirely determined by the configuration inside~$B_r({x_i})$. 
Thus, by Lemma~\ref{lem:mixing_IIC}, 
\begin{align}
 \Big|\phi_\delta\Ll[\mathcal A_{F_0(\cdot\,-x_i)}(\calL_i)\one_{E_{x_i}}\,\big|\,\calL_i \cap \calL_{X,\eps}^{\rm odd}=\emptyset \text{ and }\calL_{\rm out} = L_{\rm out}\Rr] - \phi_\delta\left[\mathcal A_{F_0}(\lpcfg)\one_{E_0}\,\big|\, \calL_{\{0\},\eps}^{\rm odd}=\emptyset\right]\Big|\le C\left(\tfrac \eps \dst\right)^{c},
\end{align}
for potentially altered values of the constants~$C, c > 0$. 
Additionally,~\eqref{eq:h1c} gives
\begin{equation*}
\Big|\phi_\delta[\mathcal A_{F_0}(\calL)\one_{E_0}|\calL_{\{0\},\eps}^{\rm odd}=\emptyset]-\bc^\delta(\eps)\Big|\le  C\left(\tfrac\eps \dst\right)^{c}.
\end{equation*}
Combining the two previous displays, we conclude that 
\begin{align}\label{eq:h1e}
\Big|\phi_\delta\Ll[\mathcal A_{F_0(\cdot\,-x_i)}(\calL_i)\one_{E_{x_i}}\,\big|\,\calL_i \cap \calL_{X,\eps}^{\rm odd}=\emptyset \text{ and }\calL_{\rm out} = L_{\rm out}\Rr] - \bc^\delta(\eps)\Big|
&\le2  C\left(\tfrac\eps \dst\right)^{c}.
\end{align}
Finally, also notice that 
\begin{align}
\phi_\delta\Ll[\Ll(\prod_{i=1}^k   \phi_\delta\Ll[\calL_i \cap \calL_{X,\eps}^{\rm odd}=\emptyset\,\big|\,\calL_{\rm out} = L_{\rm out}\Rr] \Rr)\mathcal A_F(\calL_{\rm out})\one_{E_{\rm out}}\Rr]
& = \phi_\delta[\mathcal A_F(\calL_{X,\eps})\one_{E_{\rm out}}] \qquad \text{ and } \\
\prod_{i=1}^k   \phi_\delta\Ll[\calL_i \cap  \calL_{X,\eps}^{\rm odd}=\emptyset\,\big|\,\calL_{\rm out} = L_{\rm out}\Rr]  &\leq C^k \pi_1^\delta(\eps,\dst)^{k}.
\end{align}

Thus, inserting~\eqref{eq:h1e} and~\eqref{eq:h1d} into~\eqref{eq:a} and averaging over all the possible realizations of~$L_{\rm out}$, we find that 
\begin{align}
\Big|\phi_\delta\Ll[\mathcal A_F(\calL)\one_{E_{\rm out}\cap E_{x_1}\cap \dots\cap E_{x_k}}\Rr] - \bc^\delta(\eps)^k \phi_\delta[\mathcal A_F(\calL_{X,\eps})\one_{E_{\rm out}}] \Big| 
\leq C^k \pi_1^\delta(\eps,\dst)^{k} \left(\tfrac\eps \dst\right)^{c}
 \end{align}
 with an increased value of~$C$. 

Using~\eqref{eq:hh_23} we conclude from the above that 
\begin{align}
\Big|\phi_\delta\Ll[\mathcal A_F(\calL)\Rr] - \bc^\delta(\eps)^k \phi_\delta[\mathcal A_F(\calL_{X,\eps})] \Big| 
\leq (2k+1) C^k \pi_1^\delta(\eps,\dst)^{k} \left(\tfrac\eps \dst\right)^{c},
 \end{align}
as required. 
\end{proof}

\section{The one-arm exponent and point-to-point correlations}\label{sec:one_arm}

For~$\eps>0$,~$x\in\R^2$ and~$y\notin B_{2\eps}(x)$, set
\begin{equation*}
 \widetilde \pi_1^\delta(x,y,\eps):=\phi_\delta[\calL_{\{x,y\},\eps}^{\rm odd}=\emptyset].
 \end{equation*}
In Proposition~\ref{prop:one arm}, the quantity above will be related to the GFF scaling limit using~\eqref{eq:Monday's formula} for a specific function~$F \in \calF_\eps$. Separately, we will relate~$ \widetilde \pi_1^\delta(x,y,\eps)$ to one-arm events in the random-cluster model. 
 

\begin{proposition}\label{prop:one arm}
There exists a constant~$c>0$ such that the following holds. 
For every~$\dst>0$, there exists a constant~$C = C(\dst)>0$ such that, 
for every~$\eps \in(0,d/2)$, every~$x\in\R^2$ satisfying~$\dst\leq |x|\leq 1/\dst$, and for sufficiently small~$\delta>0$,
\begin{align}\label{eq:prop_one_arm}
\Ll|\bc^\delta(\eps)^{2}\widetilde \pi_1^\delta(0,x,\eps)-a_{\eps}(x)\left(\tfrac\eps{|x|}\right)^{1/4}\Rr|\le C \pi_1^\delta(\eps,|x|)^2 \left(\tfrac\eps{|x|}\right)^c + o_\delta(1),
\end{align}
where~$a_{\eps}(x)>0$ is a quantity that converges to a constant~$a_0>0$ independent of~$x$ when~$\eps \to 0$
and~$o_\delta(1)$ is a quantity converging to~$0$ as~$\delta \to 0$. 
Moreover, both convergences are uniform over~$x$ satisfying~$\dst\leq |x|\leq 1/\dst$. 
\end{proposition}

\begin{proof}
Fix~$d>0$ and~$\eps \in (0,d/2)$. 
For~$x\in\R^2$ satisfying~$\dst\leq |x|\leq 1/\dst$, consider the test function~$F = \frac{1}{2\eps^2}\one_{B_\eps(0)} - \frac{1}{2\eps^2}\one_{B_\eps(x)}$. 
The right-hand side in~\eqref{eq:Monday's formula} for this function~$F$ may be computed explicitly and yields
\begin{align}\label{e.GFF_comput1}
    \exp\big( \tfrac{1}{2\pi^2} \iint_{\R^2\times\R^2} F(z)F(z')\log|z-z'| \d z\d z' \big) = a_{\eps}(x)\big(\tfrac\eps{|x|}\big)^{1/4},
\end{align}
where 
$a_\eps(x)=e^{2b_0-2b_\eps(x)}$, with
\begin{align}\label{eq:uu}
    b_0  = \tfrac{1}{8\pi^2}\!\!\iint_{B_1(0)\times B_1(0)}\!\!\log|\rz-\rz'|\d \rz\d \rz'
    \, \text{ and } \, 
    b_\eps(x)  = \tfrac{1}{8\pi^2}\!\!\iint_{B_1(0)\times B_1(0)}\!\!\log\tfrac{|x+\eps(\rz-\rz')|}{|x|}\d\rz\d\rz'.
\end{align}

By Theorem~\ref{thm:Monday's formula}, 
\begin{align}\label{eq:monday_consequ_pi1}
	\lim_{\delta\to0}\phi_\delta[\mcl A_F(\mcl L)] = a_{\eps}(x)\left(\tfrac\eps{|x|}\right)^{1/4}.
\end{align}
Note the implicit dependence on~$x$ of the left-hand side above via its use in the definition of~$F$. 
Thus, when treating~$x$ as a variable,~\eqref{eq:monday_consequ_pi1} is a pointwise convergence.
It may be upgraded to a uniform convergence over the set~$\{x \in \bbR^2:  \dst\leq |x|\leq 1/\dst\}$ using the Ascoli--Arzel\`a theorem.
Indeed, \cite[Thm. 4.5]{DumKozLamMan26} proves that the functions~$x \mapsto \phi_\delta[\mcl A_F(\mcl L)]$ for~$\delta > 0$ small enough are equicontinuous over this set.  
%
%
%
%
%
%

It is straightforward to see
\begin{align*}
\phi_\delta[\mathcal A_F(\calL_{\{0,x\},\eps})]=\phi_\delta[\calL^{\rm odd}_{\{0,x\},\eps}=\emptyset]=\widetilde\pi_1^\delta(0,x,\eps).
\end{align*}
Using the above,~\eqref{eq:monday_consequ_pi1} and Proposition~\ref{prop:forget small loops}, we obtain~\eqref{eq:prop_one_arm}.
The convergence of~$a_\eps(x)$ as~$\eps$ tends to $0$ is obtained from its explicit form~\eqref{eq:uu}. 
\end{proof}

We now transfer the information from~$\widetilde\pi_1^\delta(x,y,\eps)$ to connection probabilities in the random-cluster model. This part only relies on~\eqref{eq:RSW}  and its implications.

\begin{corollary}\label{c.two_point}
Under the same setup as in Proposition~\ref{prop:one arm}, we have 
\begin{align}\label{eq:b1}
\Ll|\bc^\delta(\eps)^{2}\phi_\delta[B_\eps(0)\longleftrightarrow B_\eps(x)]-\tfrac{a_{\eps}(x)}{2}\big(\tfrac\eps{|x|}\big)^{1/4}\Rr|&\le C\pi^\delta_1(\eps,|x|)^2\big(\tfrac\eps{|x|}\big)^c+o_\delta(1).
\end{align}
As a consequence, there exist~$c',C'>0$ such that for every~$x\neq 0$ and every~$\eps>0$ sufficiently small,
\begin{equation}\label{eq:b2}
c'\big(\tfrac\eps{|x|}\big)^{1/8}\le \liminf_{\delta\rightarrow0}\pi_1^\delta (\ep,|x|)\le\limsup_{\delta\rightarrow0}\pi_1^\delta (\ep,|x|)\le C'\big(\tfrac\eps{|x|}\big)^{1/8}.
\end{equation}
\end{corollary}

\begin{proof}
We focus first on the proof of~\eqref{eq:b1}.
When~$\calL_{\{0,x\},\eps}^{\rm odd}=\emptyset$,~$B_\eps(0)$ is connected to~$B_\eps(x)$ in the primal or dual configurations. We deduce that 
\begin{align*}
\left|\widetilde\pi_1^\delta(0,x,\eps)-\phi_\delta[B_\eps(0)\longleftrightarrow B_\eps(x)]-
\phi_\delta[B_\eps(0)\stackrel{*}\longleftrightarrow B_\eps(x)]\right|
= \phi_\delta\left[\begin{array}{l}B_\eps(0)\stackrel{*}\longleftrightarrow B_\eps(x)\\ B_\eps(0)\longleftrightarrow B_\eps(x)\end{array}\right].
\end{align*}
By self-duality, \eqref{eq:RSW} and the FKG inequality, there are constants~$c,C=C(d)>0$ independent of~$\delta  < \eps$ and~$x$ as above such that
\begin{align}
\Big|\phi_\delta[B_\eps(0)\longleftrightarrow B_\eps(x)]-&\phi_\delta[B_\eps(0)\stackrel{*}\longleftrightarrow B_\eps(x)]\Big|\le C\delta^c, \label{e.delta^c} \\
\phi_\delta\left[\begin{array}{l}B_\eps(0)\stackrel{*}\longleftrightarrow B_\eps(x)\\ B_\eps(0)\longleftrightarrow B_\eps(x)\end{array}\right]&\le \phi_\delta[B_\eps(0)\longleftrightarrow B_\eps(x)]\phi_\delta[B_\eps(0)\stackrel{*}\longleftrightarrow B_\eps(x)] \\
&\le C\left(\tfrac\eps{|x|}\right)^c\phi_\delta[B_\eps(0)\longleftrightarrow B_\eps(x)] \leq C^2 \pi_1^\delta(\eps,|x|)^2 \left(\tfrac\eps{|x|}\right)^c.
\end{align}
Combining these, we find that
\begin{align*}
\left|\widetilde\pi_1^\delta(0,x,\eps)-2\phi_\delta[B_\eps(0)\longleftrightarrow B_\eps(x)]\right| \leq C^2 \pi_1^\delta(\eps,|x|)^2 \left(\tfrac\eps{|x|}\right)^c+ C\delta^c.
\end{align*}
Inserting the above into Proposition~\ref{prop:one arm}, we obtain~\eqref{eq:b1}.

We turn to the proof of~\eqref{eq:b2}.
As a consequence of~\eqref{eq:RSW} (see for instance \cite{duminil2022planar}),
\begin{equation*}
	\phi_\delta[B_\eps(0)\longleftrightarrow B_\eps(x)]\asymp \pi_1^\delta (\ep,|x|)^2,
\end{equation*}
uniformly in~$\delta$ small enough. 
This, together with~\eqref{eq:b1}, readily implies~\eqref{eq:b2}.
\end{proof}

\begin{proof}[Proof of~\eqref{eq:a1}]
Fix a point~$x$ with~$|x|=1$. Substitute~$N^{-1}$ for~$\delta$ in~\eqref{eq:b2} and take~$N \to \infty$. 
Then, Corollary~\ref{c.two_point} gives constants~$c',C'>0$ such that, for sufficiently small~$\eps$,
\begin{align*}
    c'\eps^{1/8}\le \liminf_{N\to\infty}\pi_1 (\ep N,N)\le\limsup_{N\to\infty}\pi_1(\ep N,N)\le C'\eps^{1/8}.
\end{align*}
The quasi-multiplicativity~\eqref{eq:quasi_pi1} of~$\pi_1$ allows us to apply Lemma~\ref{l.exp_quasi_multi} and  deduce~\eqref{eq:a1}. \end{proof}

\begin{proof}[Proof of Proposition~\ref{prop:limit_two_pt}]
Fix~$x \in \bbR^2 \setminus \{0\}$ as well as~$\eps > \delta > 0$. All constants below are uniform in~$\eps$,~$\delta$ and~$x$ in compacts of~$\bbR^2 \setminus \{0\}$.
Write
\begin{equation*}
\phi_\delta[0\longleftrightarrow x_\delta]
=\phi_\delta[0\longleftrightarrow x_\delta \,|\,B_\eps(0)\longleftrightarrow B_\eps(x)]\phi_\delta[B_\eps(0)\longleftrightarrow B_\eps(x)].
\end{equation*}

We claim the existence of constants~$C,c\in(0,\infty)$ such that
\begin{equation}\label{e.|-|<Ceps^c}
	\Big|\phi_\delta[0\longleftrightarrow x_\delta\,|\,B_\eps(0)\longleftrightarrow B_\eps(x)]-g_\eps(\delta)^2\Big|
	\le C\eps^c \phi_\delta[0\longleftrightarrow x_\delta\,|\,B_\eps(0)\longleftrightarrow B_\eps(x)],
\end{equation}
where 
\begin{equation}\label{e.g_eps(delta)=}
g_\eps(\delta):=\phi_\delta[0\longleftrightarrow \infty\,|\,B_\eps(0)\longleftrightarrow\infty] \asymp \pi_1(\eps /\delta).
\end{equation}
Equation~\eqref{e.|-|<Ceps^c} is obtained by a mixing property similar to Lemma~\ref{lem:mixing_IIC}, but where the conditioning is on having a primal arm rather than~$\calL_{\{0\},\eps}^{\rm odd}=\emptyset$ (which is equivalent to the existence of a primal or dual arm). 
The exact mixing result was formulated in~\cite{basu2017kesten, garban2018scaling, kesten1986incipient, jarai2003invasion}. We will not detail it here.

Overall, we deduce that 
\begin{equation}\label{e.g_eps(delta)=1}
\Big|\phi_\delta[0\longleftrightarrow x_\delta]-g_\eps(\delta)^2\phi_\delta[B_\eps(0)\longleftrightarrow B_\eps(x)]\Big|\le C\eps^c\phi_\delta[0\longleftrightarrow x_\delta].
\end{equation}

By Corollary~\ref{c.two_point}, for any fixed~$\eps > 0$, 
\begin{equation*}
	\lim_{\delta\to0}\frac{2\bc^\delta (\eps)^2}{\eps^{1/4}a_\eps(x)}\phi_\delta[B_\eps(0)\longleftrightarrow B_\eps(x)] = |x|^{-1/4}.
\end{equation*}
From the above and~\eqref{eq:a1}, it follows that we may choose~$\eps (\delta) \geq \delta$ for each~$\delta > 0$ with $\eps(\delta)$ tending to $0$ as $\delta$ tends to $0$, and
\begin{gather*}
    \pi_1(\eps(\delta) /\delta) =\delta^{o(1)} \,\big(\tfrac{\delta}{\eps(\delta)}\big)^{1/8},
    \\
    \lim_{\delta\to0}\frac{2\bc^\delta (\eps(\delta))^2}{\eps(\delta)^{1/4}a_{\eps(\delta)}(x)}\phi_\delta\big[B_{\eps(\delta)}(0)\longleftrightarrow B_{\eps(\delta)}(x)\big]= |x|^{-1/4}.
\end{gather*}
Furthermore, the choice of~$\eps$ and the last convergence above may be uniform in~$x$ in compacts of~$\bbR^2 \setminus \{0\}$. 

Set 
\[
f(\delta):=g_{\eps(\delta)}(\delta)\, \left(\frac{2\bc^\delta (\eps(\delta))^2}{\eps(\delta)^{1/4} a_0 }\right)^{-1/2},
\]
where~$a_0 = \lim_{\delta\to0}a_{\eps(\delta)}(x)$ is given in Proposition~\ref{prop:one arm} and is independent of~$x$. 
Combining the previous two displays with~\eqref{e.g_eps(delta)=1} gives 
\begin{align}\label{eq:limit_two_pt2}
    \lim_{\delta\to 0} f(\delta)^{-2}\rcE_{\delta}[0\longleftrightarrow x_\delta]
  =  \lim_{\delta\to 0}   \big(\tfrac{ g_{\eps(\delta)}(\delta)}{f(\delta)}\big)^2\phi_\delta\big[B_{\eps(\delta)}(0)\leftrightarrow B_{\eps(\delta)}(x)\big]
   = |x|^{-1/4},
\end{align}
as required.
Finally, using the estimate on~$g_\eps(\delta)$ in~\eqref{e.g_eps(delta)=} together with the properties of~$\eps(\delta)$, we conclude that~$f(\delta)=\delta^{1/8+o(1)}$.

Note that~$f$ may appear to depend on the choice of~$\delta \mapsto\eps(\delta)$ above, which in turn depends on the compact set in which~$x$ is chosen. However, by considering~\eqref{eq:limit_two_pt2}
for the fixed point~$x= (1,0)$, we conclude that all definitions of~$f$ coincide. 
\end{proof}

\section{Two-arm exponent}\label{sec:two_arms} 

For~$x,y\in\mathbb R^2$ and~$\eps \geq \delta > 0$, consider the following sets of loops: 
\begin{align}
\calL^{\rm odd}&:=\calL_{\{0,x,y,x+y\},\eps}^{\rm odd}, \notag\\
\calL^2&:=\{\ell\in\calL_{\{0,x,y,x+y\},\eps}\text{ surrounding exactly 2 of the balls~$B_\eps(0)$,~$B_\eps(y)$,~$B_\eps(x)$,~$B_\eps(x+y)$}\},\label{e.calN=}
\\
\calL^2_1&:=\{\ell\in \calL^2\text{ surrounding exactly one of the balls~$B_\eps(0)$ or~$B_\eps(y)$}\}, \label{e.calN_0=}
\end{align}
as well as, for~$k\geq1$, the probabilities 
\begin{align}
\widetilde\pi_{2k}^\delta(x,y,\eps)&:=\phi_\delta[|\calL^2_1|\ge k,\calL^{\rm odd}=\emptyset],\label{e.pi_2(x,y,eps,delta)=}
\\
\widetilde\Delta^\delta(x,y,\eps)&:=
\phi_\delta\left[(-1)^{|\calL^2|}\one_{\{\calL^2_1=\calL^{\rm odd}=\emptyset\}}\right] \label{e.Delta(x,y,eps,delta)=}.
\end{align}
Finally, for~$R > r >0$ define the~$2k$-arm probability as 
\begin{align}\label{eq.2k_arm_event}
\pi_{2k}^\delta(r,R)=\phi_\delta[ \text{$\Lambda_r^\delta$ is connected to~$(\Lambda_R^\delta)^c$ by~$k$ disjoint clusters of~$\Ann(r,R)$}].
\end{align}
The quasi-multiplicativity of Proposition~\ref{p.pi_quasi-multi} also applies to~$\pi_{2k}$, with constants depending on~$k$. 

\begin{figure}[H]
    \centering
    \includegraphics[width=0.6\textwidth]{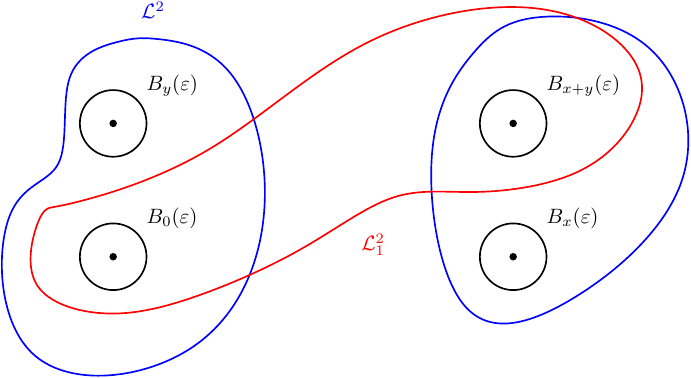}
    \caption{The blue and red loops are typical loops in~$\calL^2$ and~$\calL^2_1$, respectively.}
    \label{fig:four_ball}
\end{figure}

The quantities $\widetilde\pi_{2k}^\delta$ and $\widetilde\Delta^\delta$ will be estimated using the GFF via Theorem~\ref{thm:Monday's formula} 
and will be related to $\pi_{2k}$ and $\Delta$, respectively. 
We start with the following lemma which captures the link between~$\widetilde\pi_{2k}$ and~$\pi_{2k}$.

\begin{lemma}\label{l.tildePi2k_estimate}
Fix~$k \ge 1$. For any~$\eps \geq \delta > 0$ and~$x,y \in \bbR^2$ with~$|x|/16 \geq |y| \geq 4\eps$,
\begin{align}
\widetilde\pi_{2k}^\delta(x,y,\eps) \asymp \pi_{2k}^\delta(|x|,|y|)^2\pi_1^\delta(\eps,|y|)^4.
\end{align}
\end{lemma}

\begin{proof}[Proof of Lemma~\ref{l.tildePi2k_estimate}]
Fix~$k$,~$\eps$,~$\delta$,~$x$ and~$y$ as above and let~$r = |y|$ and~$R = |x|$. 
All constants below are allowed to depend on~$k$, but not on any of the other quantities fixed above. 

We start with the lower bound on~$\widetilde\pi_{2k}^\delta(x,y,\eps)$.
Observe that, for any configuration contributing to~$\widetilde\pi_{2k}^\delta(x,y,\eps)$ the following events need to occur (see also Figure~\ref{fig:pikexplain.pdf}):
\begin{itemize}
\item for each~$v \in \{0,x,y,x+y\}$,~$\Lambda_\eps^\delta(v)$ is connected to~$\Lambda_{r/4}^\delta(v)^c$ by either a primal or dual path;
\item $\Lambda_{4r}^\delta$ is connected to~$(\Lambda_{R/4}^\delta)^c$ by~$k$ disjoint clusters of~$\Ann(r,R/4)$;
\item $\Lambda_{4r}^\delta(x)$ is connected to~$\Lambda_{R/4}^\delta(x)^c$ by~$k$ disjoint clusters of~$\Ann_x(r,R/4)$.
\end{itemize}
Write~$E$ for the intersection of the events above. 
Observe that they are defined in terms of disjoint parts of the space, separated by annuli of outer radius twice the inner radius. 
It follows from~\eqref{eq:mixing} that
\begin{align}
	\widetilde\pi_{2k}^\delta(x,y,\eps) \leq \phi_\delta(E)  
	\les   \pi_{2k}^\delta(4r,R/4)^2\pi_1^\delta(\eps,r/4)^4 
	\les   \pi_{2k}^\delta(r,R)^2\pi_1^\delta(\eps,r)^4.
\end{align}
The second inequality is due to Proposition~\ref{p.pi_quasi-multi}.

\begin{figure}[h]
    \centering
    \includegraphics[width=0.7\linewidth]{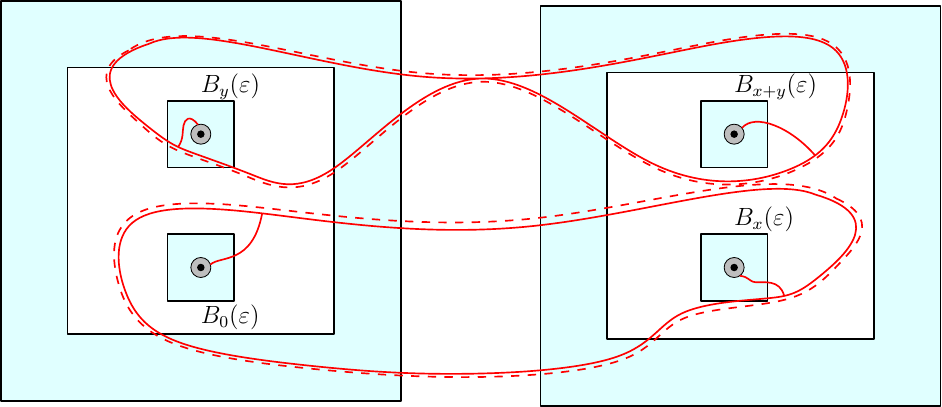}
    \caption{A configuration contributing to~$\widetilde\pi_{2k}^\delta(x,y,\eps)$ with~$k=2$. There are two loops surrounding two balls, with one centred at~$0$ or~$y$ and the other at~$x$ or~$x+y$. 
    As there are no loops surrounding a single ball, each ball is connected to one of the loops above by either a primal or dual path. 
    The annuli in the definition of~$E$ are shaded blue. }
    \label{fig:pikexplain.pdf}
\end{figure}

Conversely, it is a consequence of~\eqref{eq:RSW} and the separation-of-arms theory that, if~$E$ occurs, the different arms may be connected with positive probability to produce configurations contributing to~$\widetilde\pi_{2k}^\delta(x,y,\eps)$. We will not detail this step as it is tedious but classical; see~\cite{Kes87a, Nol08, duminil2022planar}. We conclude that
\begin{align}
	\widetilde\pi_{2k}^\delta(x,y,\eps) 
	\ges  \phi_\delta(E)  
	\ges  \pi_{2k}^\delta(4r,R/4)^2\pi_1^\delta(\eps,r/4)^4 
	\ges  \pi_{2k}^\delta(r,R)^2\pi_1^\delta(\eps,r)^4.
\end{align}
The second inequality is due to~\eqref{eq:mixing} and the last due to the inclusion of events.
\end{proof}

Similarly to Proposition~\ref{prop:one arm}, we relate~$\widetilde\pi_2^\delta(x,y,\eps)$ to a quantity computed in the GFF scaling limit via Theorem~\ref{thm:Monday's formula}. 
Contrary to Proposition~\ref{prop:one arm}, we do not claim any uniformity in~$x$ and~$y$ below, as this is not necessary for our purposes. 

\begin{proposition}\label{prop:1}
There exist constants~$C,c>0$ such that for every~$x,y\in \mathbb R^2$ satisfying~$\dst:=\min\{|x|,|y|,|x-y|, |x+y|\}>0$ and~$0<\delta<\eps<\dst/2$, we have 
\begin{align}\label{eq:pi_2+Delta}\qquad
\Ll|\bc^\delta(\eps)^4\Ll(\widetilde\pi_2^\delta(x,y,\eps)+\widetilde\Delta^\delta(x,y,\eps)\Rr)
-a_\eps(x,y)\eps^{1/2}\left(\tfrac{ |y|^2}{|x|^2|x-y|\,|x+y|}\right)^{1/4}\Rr|&\le   C\pi_1^\delta(\eps,\dst)^4\left(\tfrac{\eps}{\dst}\right)^c + o_\delta(1),
\end{align}
where~$a_{\eps}(x,y)>0$ is a quantity that converges to an absolute constant~$a_0>0$ independent of~$x,y$ as~$\eps$ tends to $0$
and~$o_\delta(1)$ is a quantity converging to~$0$ as~$\delta$ tends to $0$. 
\end{proposition}

\begin{proof}
Consider the test function~$F = \frac{1}{2\eps^2}\one_{B_\eps(0)}+\frac{1}{2\eps^2}\one_{B_\eps(y)} - \frac{1}{2\eps^2}\one_{B_\eps(x)}-\frac{1}{2\eps^2}\one_{B_\eps(x+y)}$.
We claim that 
\begin{equation}\label{eq:A_F_delta}
\phi_\delta[\mathcal A_F(\calL_{\{0,x,y,x+y\},\eps})]=\widetilde\Delta^\delta(x,y,\eps)+\widetilde\pi_2^\delta(x,y,\eps).
\end{equation}
Indeed (see Figure~\ref{fig:four_ball}),
\begin{equation*}
\cos\left(\int_{\itr(\lo)}F(z)\d z\right)
=\begin{cases} 
\ 0&\text{ if }\ell\in \calL^{\rm odd},\\
\ 1&\text{ if }\ell\in \calL^2_1,\\
\ -1&\text{ if }\ell \in \calL^2\setminus\calL^2_1.
\end{cases}
\end{equation*}
As a consequence,~$\mathcal A_F(\calL^2_1)=1$ and~$\mathcal A_F(\calL^2\setminus \calL^2_1)=(-1)^{|\calL^2\setminus \calL^2_1|}$. 
Moreover, the topology of loops implies that~$\{\calL^2_1\ne\emptyset\}$ and~$\{\calL^2\setminus\calL^2_1 \ne \emptyset\}$ are disjoint. 
These observations and the definitions of~$\widetilde\Delta^\delta(x,y,\eps)$ and~$\widetilde\pi_2^\delta(x,y,\eps)$ lead to~\eqref{eq:A_F_delta}.

Computing the right-hand side in~\eqref{eq:Monday's formula} for the above function~$F$, we conclude that 
\begin{align}\label{e.GFF_comput2}
   \lim_{\delta \to 0} \phi_\delta[\mcl A_F(\mcl L)]
    = a_\eps(x,y)\tfrac{\eps^\frac{1}{2}|y|^\frac{1}{2}}{|x|^\frac{1}{2}|x-y|^\frac{1}{4}|x+y|^\frac{1}{4}},
\end{align}
where
$a_\eps(x,y) =e^{4b_0+2b_\eps(x)-2b_\eps(y)-b_\eps(x+y)-b_\eps(x-y)}$, with~$b_0$ and~$b_\eps(.)$ given by~\eqref{eq:uu}.
In particular~$a_\eps(x,y)  $ tends to $e^{4b_0} =: a_0 >0$ as $\eps$ tends to $0$. 

Finally, using  Proposition~\ref{prop:forget small loops} to relate~$\phi_\delta[\mathcal A_F(\calL)]$ to~$\phi_\delta[\mathcal A_F(\calL_{\{0,x,y,x+y\},\eps})]$
we obtain the desired conclusion. 
%
\end{proof}

To separate the contribution of~$\widetilde\pi_2^\delta(x,y,\eps)$ and~$\widetilde\Delta^\delta(x,y,\eps)$ to~\eqref{eq:pi_2+Delta}, we will bound their difference. 

\begin{proposition}\label{prop:2}
There exist constants~$C,c>0$ such that for every~$\eps \geq \delta >0$ and every~$x,y\in \mathbb R^2$ with 
$|x|/16 \geq |y| \geq 4\eps$,
\begin{align*}
 \Ll|\widetilde\pi_2^\delta(x,y,\eps)-\widetilde\Delta^\delta(x,y,\eps)\Rr|
 \le C\, \widetilde\pi_2^\delta(x,y,\eps)\Ll(\tfrac{|y|}{|x|}\Rr)^c+C\,\pi_1^\delta(\eps,|y|)^4\Ll(\tfrac{\eps}{|y|}\Rr)^c   + o_\delta(1).
\end{align*}
\end{proposition}

\begin{proof}
Consider the test function
$
    \widetilde F = \tfrac{1}{2\eps^2}\one_{B_\eps(0)}+\tfrac{1}{2\eps^2}\one_{B_\eps(y)} + \tfrac{1}{2\eps^2}\one_{B_\eps(x)}+\tfrac{1}{2\eps^2}\one_{B_\eps(x+y)}.
$ 
Then
\begin{align}
\cos\left(\int_{\itr(\lo)}\widetilde F(z)\d z\right)=\begin{cases} 
\ 0&\text{ if }\ell\in \calL^{\rm odd},\\
\ -1&\text{ if }\ell\in \calL^2,\\
\ 1&\text{ if }\ell\in \calL_{\{0,x,y,x+y\},\eps} \setminus (\calL^{\rm odd} \cup \calL^{2}),
\end{cases}
\end{align}
Using again that~$\{\calL^2_1\ne\emptyset\}$ and~$\{\calL^2\setminus\calL^2_1 \ne \emptyset\}$ are disjoint, we conclude that 
\begin{align*}
	\phi_{\delta }[\mathcal A_{\widetilde F}(\calL_{\{0,x,y,x+y\},\eps} )]=\widetilde\Delta^\delta(x,y,\eps)-\widetilde\pi_2^\delta(x,y,\eps)+2\phi_\delta\Ll[|\calL^2_1|\in 2\mathbb Z_{>0},\, \calL^{\rm odd}=\emptyset\Rr].
\end{align*}
Observe that 
\begin{equation*}
\phi_\delta\big[|\calL^2_1|\in 2\mathbb Z_{>0},\,\calL^{\rm odd}=\emptyset\big]
\le \widetilde\pi_4^\delta(x,y,\eps)\le C\,\widetilde\pi_2^\delta(x,y,\eps)\lp\tfrac{|y|}{|x|}\rp^c,
\end{equation*}
for constants~$c,C > 0$. Thus 
\begin{align}\label{eq:2deltapi}
	\big|\widetilde\Delta^\delta(x,y,\eps)-\widetilde\pi_2^\delta(x,y,\eps)\big| \leq \big|\phi_{\delta }[\mathcal A_{\widetilde F}(\calL_{\{0,x,y,x+y\},\eps} )] \big| + C \widetilde\pi_2^\delta(x,y,\eps) \lp\tfrac{|y|}{|x|}\rp^c. 
\end{align}

Note that Theorem~\ref{thm:Monday's formula} does not apply to~$\widetilde F$ as its average is not zero. 
However, if we define~$\widetilde F_R:=\widetilde F(\cdot-(R,0))-\widetilde F(\cdot+(R,0))$, we find 
\begin{align}
	\phi_{\delta }[\mathcal A_{\widetilde F}(\calL )]^2
	=\lim_{R\to \infty}\phi_\delta[\mathcal A_{\widetilde F(\cdot-(R,0))}(\calL )\mathcal A_{-\widetilde F(\cdot+(R,0))}(\calL )]
	=\lim_{R\to \infty}\phi_\delta[\mathcal A_{\widetilde F_R}(\calL)].
\end{align}
To justify this claim, we use the mixing property~\eqref{eq:mixing}, the fact that the probability that a loop intersects both~$\Lambda_{\sqrt R}(-R,0)$ and~$\Lambda_{\sqrt R}(R,0)$ tends to~$0$ as~$R$ tends to $\infty$,
and that $\calA_{F}(L)=\calA_{-F}(L)$ for any function $F$ as cosine is an even function.
Furthermore, the convergences above are uniform in~$\delta$, since both~\eqref{eq:mixing} and the bound on the existence of large loops are uniform in~$\delta$. 

Using the uniformity of the convergence above, applying Theorem~\ref{thm:Monday's formula} to~$\widetilde F_R$, then taking~$R$ to infinity, we conclude that 
\begin{align}
	\lim_{\delta\to0}\phi_{\delta }[\mathcal A_{\widetilde F}(\calL )]^2
	=\lim_{\delta\to0}\lim_{R\to \infty} \phi_\delta[\mathcal A_{\widetilde F_R}(\calL)] 
	=\lim_{R\to \infty}\lim_{\delta\to0} \phi_\delta[\mathcal A_{\widetilde F_R}(\calL)] 
	= 0.
\end{align}
Proposition~\ref{prop:forget small loops} and Lemma~\ref{l.kappa} then imply that
\begin{align*}
\Ll|\phi_\delta[\mathcal A_{\widetilde F}(\calL_{\{0,x,y,x+y\},\eps})]\Rr|\le 
C\pi_1^\delta(\eps,|y|)^4\left(\tfrac{\eps}{|y|}\right)^c + o_\delta(1).
\end{align*}
Combine the above with~\eqref{eq:2deltapi} to obtain the desired conclusion. 
\end{proof}

We record an immediate consequence of Propositions~\ref{prop:1} and~\ref{prop:2}.

\begin{corollary}\label{c.pi_2,Delta}
There exist constants~$C,c>0$ such that for every~$\eps \geq \delta >0$ and every~$x,y\in \mathbb R^2$ with 
$|x|/16 \geq |y| \geq 4\eps$,
\begin{align}\label{eq:c1}\quad
\Ll| \widetilde\pi_2^\delta(x,y,\eps) -\tfrac{a_\eps(x,y)}{2\bc^\delta(\eps)^4}\eps^{1/2}\left(\tfrac{ |y|^2}{|x|^2|x-y|\,|x+y|}\right)^{1/4}\Rr|
	&\le   C\widetilde\pi_2^\delta(x,y,\eps)\Ll(\tfrac{|y|}{|x|}\Rr)^c+C\pi_1^\delta(\eps,|y|)^4\Ll(\tfrac{\eps}{|y|}\Rr)^c + o_\delta(1),\\
\label{eq:c2}\Ll| \widetilde\Delta^\delta(x,y,\eps)-\tfrac{a_\eps(x,y)}{2\bc^\delta(\eps)^4}\eps^{1/2}\left(\tfrac{ |y|^2}{|x|^2|x-y|\,|x+y|}\right)^{1/4}\Rr|
&\le   C\widetilde\pi_2^\delta(x,y,\eps)\Ll(\tfrac{|y|}{|x|}\Rr)^c+C\pi_1^\delta(\eps,|y|)^4\Ll(\tfrac{\eps}{|y|}\Rr)^c  + o_\delta(1) .
\end{align}
\end{corollary}

Now, we are ready to prove the asymptotic of Theorem~\ref{thm:main} for the two-arm probability.

\begin{proof}[Proof of~\eqref{eq:a2}] 
Fix~$x\in\R^2$ with~$|x|=1$ and consider~$y \in \bbR^2$ and~$\eps > \delta > 0$ so that Corollary~\ref{c.pi_2,Delta} applies. 
By Lemma~\ref{l.tildePi2k_estimate} 
\begin{align}
	\widetilde\pi_2^\delta(x,y,\eps)
	\asymp \pi_2^\delta(|y|,1)^2\pi_1^\delta(\eps,|y|)^4
\end{align}
Inserting this into~\eqref{eq:c1}, dividing by~$\pi^\delta_1(\eps,|y|)^4$, taking~$\delta$ to $0$ then~$\eps$ to $0$ and using~\eqref{eq:b2}, we conclude that 
\begin{align}
	c'|y|^{1/2} \leq \liminf_{\delta\to0}\pi_2^\delta(|y|,1) \leq \limsup_{\delta\to0}\pi_2^\delta(|y|,1) \leq C'|y|^{1/2}.
\end{align}
for constants~$c',C'$ independent of~$y$. Apply Lemma~\ref{l.exp_quasi_multi} to deduce~\eqref{eq:a2}. 
\end{proof}

\section{The influence exponent}\label{sec:Delta}

The goal of this section is to prove the asymptotic~\eqref{eq:a3} of Theorem~\ref{thm:main} for   $\Delta$. 

\subsection{Preliminaries}


Recall the quantity~$\Delta(r,R)$ introduced in~\eqref{eq:Delta_def},
its rescaled version~$\Delta^\delta$, 
and the quantity~$\widetilde\Delta^\delta$ introduced in~\eqref{e.Delta(x,y,eps,delta)=}, whose asymptotic was determined in \eqref{eq:c2}.
The following proposition relates~$\Delta^\delta$ to~$\widetilde\Delta^\delta$ and is the core of this section.

%

\begin{proposition}\label{prop:from Delta to influence}
	For every~$x,y \in \bbR^2$ and~$\eps \geq \delta > 0$ satisfying~$|x|/16 \geq |y| \geq 4\delta$,
    \begin{align}
    \frac{\widetilde\Delta^\delta(x,y,\eps)}{\pi_1^\delta(\eps,|y|)^4}\asymp \Delta^\delta(|y|,|x|)^2+o_{\eps,\delta}(1)\label{e.prop:two-arm_Delta}
    \end{align}
    where~$o_{\eps,\delta}(1)$ is a quantity that satisfies~$\lim_{\eps\to0}\limsup_{\delta\to0} \Ll|o_{\eps,\delta}(1)\Rr|=0$. 
	The constants in~$\asymp$ are uniform in~$x$,~$y$,~$\delta$ and~$\eps$. 
\end{proposition}

From this proposition, it is fairly direct to obtain~\eqref{eq:a3}.

\begin{proof}[Proof of~\eqref{eq:a3}]
Fix~$x\in\R^2$ with~$|x|=1$ and consider~$y \in \bbR^2$ and~$\eps > \delta > 0$ so that Corollary~\ref{c.pi_2,Delta} applies. 
Dividing~\eqref{eq:c2} by~$\pi_1^\delta(\eps,|y|)^4$ and using Lemma~\ref{l.tildePi2k_estimate} we find 
\begin{align}
	\Ll| \frac{\widetilde\Delta^\delta(x,y,\eps)}{\pi_1^\delta(\eps,|y|)^4} - |y|  \frac{(\eps/y)^{1/2}}{\pi_1^\delta(\eps,|y|)^4}\Big(\frac{a_\eps(x,y)}{2\bc^\delta(\eps)^4}+o_y(1)\Big)\Rr|
	\le   C \pi_2^\delta(|y|,1)^{2} |y|^c+C \Ll(\tfrac{\eps}{|y|}\Rr)^c  + o_\delta(1),
\end{align}
where~$o_y(1)$ is a quantity that does not depend on~$\eps$ or~$\delta$ and converges to~$0$ as~$y \to 0$. 
Using Proposition~\ref{prop:from Delta to influence},~\eqref{eq:a1},~\eqref{eq:a2} and taking~$\delta$ and then $\eps$ to $0$, we find that 
\begin{align}
    c'|y|^{1/2} \leq \liminf_{\delta\to0}\Delta^\delta(|y|,1) \leq \limsup_{\delta\to0}\Delta^\delta(|y|,1)  \leq C'|y|^{1/2},
\end{align}
for constants~$c',C'$ independent of~$y$ and all~$y$  with~$|y|$ small enough. Apply Lemma~\ref{l.exp_quasi_multi} to deduce~\eqref{eq:a3}. 
%
%
%
\end{proof}

The rest of the section is devoted to the proof of Proposition~\ref{prop:from Delta to influence}, 
which is, in fact, more complicated than the equivalent results for the one- and two-arm probabilities. 
This is due to~$\Delta$ being the difference of two probabilities, rather than the probability of a given event. 

At the heart of the proof lies the lemma below, which requires a few definitions before being stated. 
For~$s\in\{0,1\}$, consider the events
\begin{align}\label{e.L(s)=,R(s)=}
    L(s) := \Ll\{B_\eps(0) \stackrel{s}{\longleftrightarrow}B_\eps(y)\Rr\},\qquad R(s) := \Ll\{B_\eps(x) \stackrel{s}{\longleftrightarrow}B_\eps(x+y)\Rr\},
\end{align}
which correspond to pairs of~$\eps$-balls being connected in the primal configuration if~$s = 1$ or dual if~$s=0$. 

\begin{lemma}\label{l.P(LR)-P(LR)=Delta^2}
	For any~$\eps \geq \delta > 0$ and~$x,y \in \bbR^2$ with~$|x|/16 \geq |y| \geq 4\eps$,
	\begin{align}\label{e.l.P(LR)-P(LR)=Delta^2}
		\rcP_\delta[L(1)\cap R(1)] - \rcP_\delta[L(1)\cap R(0)]\asymp \Delta^\delta(|y|,|x|)^2\big( \pi^\delta_1(\eps,|y|)^4 + o_{\eps,\delta}(1)\big).
	\end{align}
	where~$o_{\eps,\delta}(1)$ satisfies~$\lim_{\eps\to0}\limsup_{\delta\to0} \Ll|o_{\delta}(1)\Rr|=0$.
\end{lemma}

The remainder of this section is organized as follows. 
In Section~\ref{sec:from lemma to proposition} we explain how Lemma~\ref{l.P(LR)-P(LR)=Delta^2} implies Proposition~\ref{prop:from Delta to influence}. 
The following sections are dedicated to proving Lemma~\ref{l.P(LR)-P(LR)=Delta^2}: the upper bound is presented in Section~\ref{sec:upper} and the lower bound in Section~\ref{sec:lower}.
The latter bound is more intricate and will require some coupling arguments which were developed in \cite{duminil2022planar} and will be recalled in Section~\ref{sec:basics}. 

Henceforth, fix~$\eps$,~$\delta$,~$x$ and~$y$ as in the statement of Proposition~\ref{prop:from Delta to influence}. All constants below are independent of these choices. 
We will remove the index~$\delta$ when unnecessary, in particular we write~$\rcP = \rcP_\delta$ and~$\Lambda_r(.) = \Lambda_r^\delta(.)$.


\subsection{From Lemma~\ref{l.P(LR)-P(LR)=Delta^2} to  Proposition~\ref{prop:from Delta to influence}}\label{sec:from lemma to proposition}

Below, we assume Lemma~\ref{l.P(LR)-P(LR)=Delta^2} and derive Proposition~\ref{prop:from Delta to influence}.

\begin{proof}[Proof of Proposition~\ref{prop:from Delta to influence}] 
Set     
\begin{align}\label{e.E_1=cupL(s)capR(s')}
    E &:= \bigcup_{s,s'\in\{0,1\}} \Ll( L(s)\cap R(s')\Rr), \\
{\rm Bad} &:= \big\{\text{$\exists$~$\lo \in \calL$ that  intersects at least two of~$B_\eps(0)$,~$B_\eps(y)$,~$B_\eps(x)$ and~$B_\eps(x+y)$}\big\}.
\end{align}
It is then easy to check that~$\{\mathcal{L}_0^2=\calL^{\rm odd}=\emptyset\}=E$, so 
\begin{align}\label{e.Delta(x,y,eps,delta)=phi_rewrite}
    \widetilde\Delta^\delta(x,y,\eps) = \rcP\Ll[(-1)^{|\mathcal{L}^2\setminus\mathcal{L}^2_0|}\one_E\Rr].
\end{align}

First, we claim that,  for any~$s,s'\in\{0,1\}$, 
\begin{align}\label{e.phi[LRG^C]<}
   \rcP \big[L(s)\cap R(s')\cap {\rm Bad}\big] \les \pi_2^\delta(\eps,|y|)^2\pi_1^\delta(\eps,|y|)^2 = \pi_1^\delta(\eps,|y|)^4 o_{\eps,\delta}(1).
\end{align}
Indeed, by a union bound and self-duality, it suffices to bound the probability that~$R(1)$ occurs and that there exists a loop~$\ell$ intersecting both~$B_\eps(0)$ and~$B_\eps(y)$.
When this holds, two-arm events occur in the annuli~$\Ann_0(\eps, |y| /4)$ and~$\Ann_y(\eps, |y| /4)$ and one-arm events in the annuli~$\Ann_x(\eps, |y| /4)$ and~$\Ann_{x+y}(\eps, |y| /4)$. 
By~\eqref{eq:mixing}, the joint occurrence of these events may be bounded by a constant multiple of the product of their individual probabilities. 
This provides the inequality in~\eqref{e.phi[LRG^C]<}. The equality comes from comparing the one- and two-arm probabilities. 

Outside of~${\rm Bad}$, the events~$L(s)\cap R(s')$ indexed by~$s,s'\in\{0,1\}$ are disjoint. Thus
\begin{align*}
  \rcP\Ll[(-1)^{|\mathcal{L}^2\setminus\mathcal{L}^2_0|}\one_{E\setminus {\rm Bad}}\Rr] 
    &=  \sum_{s,s'\in\{0,1\}} \rcP\Ll[(-1)^{|\mathcal{L}^2\setminus\mathcal{L}^2_0|}\one_{L(s)\cap R(s')\cap {\rm Bad}^c}\Rr].
\end{align*}
Still outside of~${\rm Bad}$,~$|\mathcal{L}^2\setminus\mathcal{L}^2_0|$ is even if and only~$L(s)\cap R(s)$ occur for some~$s\in \{0,1\}$. 
In other words, 
\begin{align*}
   \rcP\Ll[(-1)^{|\mathcal{L}^2\setminus\mathcal{L}^2_0|}\one_{E\setminus {\rm Bad}}\Rr]  
   &= \sum_{s=0,1} \Ll(\rcP\big[L(s)\cap R(s)\cap {\rm Bad}^c\big] - \rcP\big[L(s)\cap R(1-s)\cap {\rm Bad}^c\big] \Rr) .
\end{align*}

Using~\eqref{e.phi[LRG^C]<} we conclude that
\begin{align}
\frac{\widetilde\Delta^\delta(x,y,\eps)}{\pi_1^\delta(\eps,|y|)^4} 
&=  \sum_{s=0,1} \pi_1^\delta(\eps,|y|)^{-4}  \Big(\rcP\big[L(s)\cap R(s)\cap {\rm Bad}^c\big] - \rcP\big[L(s)\cap R(1-s)\cap {\rm Bad}^c\big]  +  o_{\eps,\delta}(1)\Big)\\
&=  2\pi_1^\delta(\eps,|y|)^{-4}  \Big(\rcP \big[L(1)\cap R(1) \big] - \rcP\big[L(1)\cap R(0)\big]  +  o_{\eps,\delta}(1)\Big)\\
&\asymp \Delta^\delta(|y|,|x|)^2 + o_{\eps,\delta}(1),
\end{align}
with the last line following from~\eqref{e.l.P(LR)-P(LR)=Delta^2}.
\end{proof}

\subsection{Proof of the upper bound in Lemma~\ref{l.P(LR)-P(LR)=Delta^2}}\label{sec:upper}

In this section, we will prove the existence of~$c > 0$ such that, for all 
\begin{align}\label{e.l.P(LR)-P(LR)=Delta^2_up_bd}
    \rcP[L(1)\cap R(1)] - \rcP[L(1)\cap R(0)]\leq c\, \Delta^\delta(|y|,|x|)^2\pi^\delta_1(\eps,|y|)^4.
\end{align}

The above is a fairly direct consequence of the following lemma, proved in \cite{duminil2022planar}.

\begin{lemma}\label{lem:iota0}
	There exists~$c >0$ such that the following holds. 
	For any~$R > r > \delta$ and
	any events~$A$ and~$B$ depending on the edges in~$\Lambda^\delta_{r}$ and outside~$\Lambda^\delta_{R}$, respectively, 
	\begin{align}\label{eq:iota0}
	\big|\tfrac{\rcP[A \cap B]}{\rcP[A]\rcP[B]}-1\big| \le c \Delta^\delta(r,R).
	\end{align}
	For any~$r_1,r_2 > \delta$ with~$r_1 + r_2 \leq |x|/2$, 
	and any events~$A$ and~$B$ depending on the edges in~$\Lambda^\delta_{r_1}$ and~$\Lambda^\delta_{r_2}(x)$, respectively, 
	\begin{align}\label{eq:Delta_as_mix12}
	\big|\tfrac{\rcP[A \cap B]}{\rcP[A]\rcP[B]}-1\big|\le c \Delta^\delta(r_1,|x|) \Delta^\delta(r_2,|x|).
	\end{align}
\end{lemma}

The above relates~$\Delta$ to the mixing rate. A converse bound will be stated below, in Lemma~\ref{lem:iota2}.

\begin{proof}[Proof of Lemma~\ref{lem:iota0}]
	Equation~\eqref{eq:iota0} was proved in \cite[eq.~(1.18)]{duminil2022planar}. 
	Equation~\eqref{eq:Delta_as_mix12} is obtained from~\eqref{eq:iota0} using the technique of \cite[eq.~(5.3)]{duminil2022planar}; 
	see also \cite[Remark 5.5]{duminil2022planar}. 
\end{proof}

\begin{proof}[Proof of Equation~\eqref{e.l.P(LR)-P(LR)=Delta^2_up_bd}]
We use the notation~$\les$ introduced in Section~\ref{sec:notation}, with the constants independent of~$\eps$,~$\delta$,~$x$ and~$y$ fixed previously. 

By the same argument as in~\eqref{e.delta^c},
we have~$\rcP[R(0)]=\rcP[R(1)] + o_\delta(1)$ for some error satisfying~$\lim_{\delta\to0} o_\delta(1)=0$. 
Then,
\begin{align}\label{eq:covLR}
     \rcP[L(1)\cap R(1)] - \rcP[L(1)\cap R(0)]= \sum_{s=0}^1 \big|\rcP[L(1)\cap R(s)] - \rcP[L(1)]\rcP[R(s)]\big| + o_\delta(1).
\end{align}

Let us bound the right-hand side for~$s=1$; the proof for~$s=0$ follows the same lines. 
Write~$L=L(1)$ and~$R=R(1)$ for brevity. 
We may not directly apply~\eqref{eq:Delta_as_mix12} as~$L$ and~$R$ are not local events. 
To overcome this, we decompose them over different scales, depending on how far one needs to go to ``measure'' those events. Set the number of scales as
\begin{align}\label{e.K=logeta^-1/log2}
    K = \left\lfloor\frac{\log (|x|/|y|)}{\log 2}\right\rfloor-2.
\end{align}
Consider the events~$L_k$ for~$0\le k \le 2K+1$ defined by
\begin{align}\label{e.L_k=}
L_k= \begin{cases}
    \Ll\{\text{$B_\eps(0)\longleftrightarrow B_\eps(y)$ in~$\Lambda_{2|y|}$}\Rr\}, & \text{if } k=0;
    \\
    \Ll\{\text{$B_\eps(0) \longleftrightarrow B_\eps(y)$ in~$\Lambda_{2^{k+1}|y|}$ but not in~$\Lambda_{2^k |y|}$}\Rr\}, & \text{if }k\in\{1,\dots,K\};
    \\
    \Ll\{\text{$B_\eps(0) \longleftrightarrow B_\eps(y)$ in~$\Lambda_{2^{2K-k}|y|}(x)^{\comple}$ but not in~$\Lambda_{2^{2K-k+1}|y|}(x)^\comple$}\Rr\}, & \text{if }k\in\{K+1,\dots,2K\};
    \\
    \Ll\{\text{$B_\eps(0) \longleftrightarrow B_\eps(y)$ but not in~$\Lambda_{2|y|}(x)^\comple$}\Rr\}, & \text{if }k=2K+1.
\end{cases}
\end{align}
Define~$R_{k}$ for~$k\in\{0,\dots, 2K+1\}$ similarly, with the roles of~$0$ and~$x$ reversed.
Then, the events~$L_k$ form a partition of~$L$, and similarly for~$R$. 
In particular
\begin{align}\label{e.L cap R}
    \rcP[L\cap R]-\rcP[L]\rcP[R]=\sum_{k,\ell=0}^{2K+1} \rcP[L_k\cap R_{\ell}] - \rcP[L_k]\rcP[R_\ell].
\end{align}

For~$L_k$ to occur, one-arm events occur in~$\Lambda_{|y|/4}(0)$ and~$\Lambda_{|y|/4}(y)$. 
Moreover, when~$0 \leq k \leq K$, the four arm event (as in the definition~\eqref{eq.2k_arm_event} of~$\pi_{4}^\delta$) needs to occur in~$\Ann(2|y|,2^k |y|)$.
When~$k >K$, four-arm events occur in~$\Ann(2|y|,2^K |y|)$ and in~$\Ann_x(2^{2K-k+1}|y|,2^K |y|)$. 
Similar considerations hold for~$R_k$. 
Applying~\eqref{eq:mixing} and the quasi-multiplicativity of~$\pi_4$ (the equivalent of Proposition~\ref{eq:quasi_delta} for~$\pi_4$), 
we find 
\begin{align}
	\phi[R_k ] = \phi[L_k ]& \les
	\begin{cases}
		 \pi_1^\delta(\eps , |y|)^2 \pi_4^\delta(|y|,2^k |y|)
		 \qquad & \text{ if~$0 \leq k \leq K$}\\[.5em]
		 \pi_1^\delta(\eps , |y|)^2 \pi_4^\delta(|y|,|x|)\pi_4^\delta (2^{K-k}|x|,|x|) 
		  \qquad & \text{ if~$K < k \leq 2K$}		  
	\end{cases}		\\[.5em]
	\phi[ L_k \cap R_\ell ]& \les  \pi_1^\delta(\eps , |y|)^4 \pi_4^\delta(|y|,|x|)^2 \qquad\qquad\qquad\qquad\qquad\! \text{ if~$ k + \ell  \geq 2K$}.
\end{align}

It was proved in \cite[Theorem 1.6]{duminil2022planar} that there exist constants~$C, c > 0$ such that 
\begin{align}
	\pi_4(r,R) \leq c \Delta(r,R) (\tfrac{r}R)^c \qquad\text{ for all~$r \leq R$.}
\end{align}
Applying the above, Lemma~\ref{lem:iota0} to~$L_k$ and~$R_\ell$ with~$k+\ell < 2K$, 
using the mixing property~\eqref{eq:mixing} and the quasi-multiplicativity \eqref{eq:quasi_delta} of~$\Delta$, 
we conclude that 
\begin{align}
|\phi[L_k\cap R_{\ell}] - \phi[L_k]\phi[R_\ell]| &\les  \pi_1^\delta(\eps , |y|)^4 \Delta^\delta(|y|,|x|)^2 \,2^{-c(k+ \ell)} && \text{ if~$\ell + k < 2K$ and}\\
| \phi[L_k\cap R_{\ell}] - \phi[L_k]\phi[R_\ell] | &\les \pi_1^\delta(\eps , |y|)^4 \Delta^\delta(|y|,|x|)^2 \,2^{-2c K} && \text{ if~$\ell + k \geq 2K$.}
\end{align}
Inserting this into~\eqref{e.L cap R}, we conclude that 
\begin{align}
    \big|\phi[L(1)\cap R(1)]-\phi[L(1)]\phi[R(1)] \big|\les \pi_1^\delta(\eps , |y|)^4 \Delta^\delta(|y|,|x|)^2.
\end{align}
Together with~\eqref{eq:covLR}, the above and its equivalent version for~$R(0)$ imply~\eqref{e.l.P(LR)-P(LR)=Delta^2_up_bd}.
\end{proof}

\subsection{Crash course on coupling techniques used for the lower bound}\label{sec:basics}

In this section, we briefly describe the notion of flower domains and the coupling arguments used in \cite{duminil2022planar}. 
For more details, we refer to \cite[Sections 3 and 4]{duminil2022planar}.

\begin{definition}[Flower domain]
Fix~$R \geq 1$ and a configuration~$\omega$ on $\Z^2$. The \emph{inner flower domain} on~$\Lambda_R$ is obtained as follows. 
Write~$\calI$ for the union of all interfaces\footnote{Here, we consider interfaces as the closure of the space contained between the primal and dual edges.} between the primal and dual edges contained in~$\Ann(R,3R/2)$ that start on~$\partial \Lambda_{3R/2}$; these may end on~$\partial \Lambda_{3R/2}$ or on~$\partial \Lambda_{R}$. Write~$\calF_{\rm in}$ for the closure of the connected component of~$\Lambda_R$ in~$\bbR^2 \setminus \calI$. 

The \emph{outer flower domain} on~$\Lambda_{2R}$ is the closure of the infinite connected component~$\calF_{\rm out}$ of~$\bbR^2 \setminus \calJ$, 
where~$\calJ$ denotes the union of all interfaces in~$\Ann(3R/2,2R)$ that start on~$\partial \Lambda_{3R/2}$.

\end{definition}
%
%

%
%
%

The boundaries of~$\calF_{\rm in}$ and~$\calF_{\rm out}$ may be formed either of a circuit of primal or dual edges, 
or of an even number of alternating paths of the primal and dual graphs. 
In the latter case, we call these paths {\em petals} and index them~$P_1,\dots, P_{2k}$ with~$P_1$ being primal. 
We say that~$\calF_{\rm in}$ and~$\calF_{\rm out}$ are {\em well-separated} when they are of the first type, 
or when their petals have endpoints at distances at least~$R/2$ from each other. 
See Figures~\ref{fig:iota2} and~\ref{fig:iota3} for examples. 

A boundary condition~$\xi$ on~$\calF_{\rm in}$ or~$\calF_{\rm out}$ is said to be \emph{compatible} 
if all vertices of any primal petal are wired together and all vertices of dual petals are wired to no other vertex of the boundary of the flower domain.
The latter is equivalent to saying that all dual vertices of any dual petal are wired together in the dual boundary condition. 
Note that for flower domains with at least four petals, there exist multiple compatible boundary conditions: they depend on the wirings {\em between} the different primal petals.

\begin{definition}[Boosting pair of boundary conditions] 
A \emph{boosting pair} of boundary conditions for an inner or outer flower domain~$\mathcal{F}$ with at least four petals 
is a pair of boundary conditions~$(\xi,\xi')$ such that
\begin{itemize}
    \item~$\xi$ and~$\xi'$ are compatible with~$\mathcal{F}$;
    \item~$\xi\leq\xi'$;
    \item there exist two primal petals of~$\mathcal{F}$ that are wired together in~$\xi'$ but not in~$\xi$.
\end{itemize}
We also call boosting pair any pair of boundary conditions~$(\zeta,\zeta')$ with~$\zeta \leq \xi \leq \xi'\leq \zeta'$ and~$(\xi,\xi')$ satisfying the conditions above. 
\end{definition}

Next, we introduce the notion of  {\em double four-petal flower domain}.
While seemingly artificial, it is useful in arguments involving couplings via decision trees, as defined below.

\begin{definition}\label{lem:DFPFD}
Fix~$R\geq 1$ and a configuration~$\omega$. We say that there exists a \emph{double four-petal flower domain} between~$\Lambda_{R}$ and~$\Lambda_{2R}$ if
\begin{itemize}
    \item~$\Fout$, the outer flower domain on~$\Lambda_{2R}$ is well-separated and has exactly four petals $\Pout_1,\dots,\Pout_4$;
    \item~$\Fin$, the inner flower domain on~$\Lambda_R$ is well-separated and has exactly four petals~$\Pin_1,\dots,\Pin_4$;
    \item~$\Pin_1$ is connected to~$\Pout_1$ and~$\Pin_3$ to~$\Pout_3$ in~$\omega\cap \Fin^\comple\cap\Fout^\comple$;
    \item~$\Pin_2$ is connected to~$\Pout_2$ and~$\Pin_4$ to~$\Pout_4$ in~$\omega^*\cap \Fin^\comple\cap\Fout^\comple$.
\end{itemize}
\end{definition}

A double four-petal flower domain is a way of almost decoupling the configuration inside~$\Lambda_R$ from the outside of~$\Lambda_{2R}$. 
Indeed, the only influence the configuration in~$\calF_{\rm out}$ has on~$\calF_{\rm in}$ is by connecting~$\Pout_1$ to~$\Pout_3$ or not. 
It is essential (and immediate from the definition) that a double four-petal flower domain may be explored without revealing any of the edges in~$\calF_{\rm in}$ or~$\calF_{\rm out}$. 
We restate \cite[Lemma~3.4]{duminil2022planar} below.

\begin{lemma}\label{l.double_flower}
   There exists~$c> 0$ such that, for any~$R$ large enough,
    \begin{equation*}
        \rcP\left[\exists\text{ a double four-petal flower domain between }\Lambda_R\text{ and }\Lambda_{2R}\right] > c.
    \end{equation*}
\end{lemma}


Next, we describe how to couple different random-cluster measures in an increasing fashion using decision trees. 
Consider a finite subgraph~$G=(V,E)$ of~$\Z^2$ with~$|E|=n$. 
Let~$(U_e)_{e\in E}$ be a set of i.i.d.~uniform random variables on~$[0,1]$. For~$e=(e_1,\dots,e_n)\in E^n$ and~$t\leq n$, we write~$e_{[t]} := (e_1,\dots,e_t)$ and~$U_{[t]} := (U_{e_1},\dots,U_{e_t})$.
We extend this notation to configurations by writing~$\omega_{[t]}:=(\omega_{e_1},\dots,\omega_{e_t})$. 
Finally, let~$G_t$ denote the subgraph obtained from~$G$ by removing the edges~$e_{[t]}$.

\begin{definition}[Decision tree, stopping time {\cite[Definition~2.4]{duminil2022planar}}]
    A \emph{decision tree} is a pair~$\mathbf{T} = (e_1, (\psi_t)_{2\leq t\leq n})$ where~$e_1\in E$ and for each~$2\leq t\leq n$, the input of~$\psi_t$ is~$(e_{[t-1]},U_{[t-1]})$ and it returns~$e_t\in E\setminus\{e_1,\dots,e_{t-1}\}$. A \emph{stopping time} for~$\mathbf{T}$ is a random variable~$\tau$ taking values in~$\{1,\dots,n,\infty\}$, for which each~$\{\tau\leq t\}$ is measurable in terms of~$(e_{[t]}, U_{[t]})$.
\end{definition}

The following statement describes the coupling procedure induced by a decision tree. It is an increasing coupling between two ordered random-cluster measures. This is a restatement of \cite[Proposition~2.6]{duminil2022planar}.

\begin{proposition}
    Consider boundary conditions~$\xi\leq \xi'$ on~$G$. 
    Let~$\mathbf{T} = \left(e_1,(\psi_t)_{2\leq t\leq n}\right)$ be a decision tree and~$(U_e)_{e\in E}$ be a set of i.i.d.~uniform random variables on~$[0,1]$ under some probability measure~$\P_\mathbf{T}$. Define~$\omega,\omega'\in\{0,1\}^E$ inductively by
    \begin{align*}
        \omega_{e_{t+1}} &:=\one\big(U_{e_{t+1}}\geq \rcP_{G_t}^{\xi_t}[e_{t+1}\text{ closed}]\big),\\
        \omega'_{e_{t+1}} &:= \one\big(U_{e_{t+1}}\geq \rcP_{G_t}^{\xi'_t}[e_{t+1}\text{ closed}]\big),\quad\forall 0\leq t<n,
    \end{align*}
where~$\xi_t$ and~$\xi'_t$ are the boundary conditions on~${G_t}$ induced by~$\omega^\xi_{[t]}$ and~$\big(\omega'_{[t]}\big)^{\xi'}$ respectively. When~$t=0$, these are~$\xi$ and~$\xi'$. 

Then, $\omega$ and~$\omega'$ have laws~$\rcP^{\xi}_{G}$ and~$\rcP^{\xi'}_{G}$, respectively, and~$\omega\leq \omega'$~$\P_\mathbf{T}$-almost surely.

The same procedure may be applied to couple increasingly~$\phi_G^{\xi}$ to~$\phi_G^{\xi}[\cdot\,|\,A]$ for an increasing event~$A$. We will also use it to couple (finite restrictions of) infinite volume measures. 
\end{proposition}

We often describe the decision tree, and hence the increasing coupling of two measures, 
as ``revealing'' edges in a certain order, potentially depending on the state of the already revealed edges in one of the configurations~$\omega$ or~$\omega'$.
Throughout, for an increasing coupling~$\bbP_\mathbf{T}$, we always refer to~$\omega\le \omega'$ as the coupled pair of configurations.

\subsection{Lower bound in Lemma~\ref{l.P(LR)-P(LR)=Delta^2}}\label{sec:lower}

In this section, we prove the existence of~$c > 0$ such that 
\begin{align}\label{e.l.P(LR)-P(LR)=Delta^2_l_bd}
    \rcP[L(1)\cap R(1)] - \rcP[L(1)\cap R(0)]\geq c \,\Delta^\delta(|y|,|x|)^2\pi^\delta_1(\eps,|y|)^4 + o_\delta(1).
\end{align}

We will use here the notation~$\ges$ introduced in Section~\ref{sec:notation}, with the constants independent of~$\eps$,~$\delta$,~$x$ and~$y$ fixed previously, or of any future fixed quantities.
As stated in  Section~\ref{sec:upper},~$ \rcP[R(0)] =  \rcP[R(1)] + o_\delta(1)$. 
Using this, the FKG inequality and~\eqref{eq:mixing}, we find
\begin{align}
    \rcP[L(1)\cap R(1)] - \rcP[L(1)\cap R(0)]
  &  \geq  \rcP[L(1)] \big(\rcP[R(1)\,|\, L(1)] -  \rcP[R(1)] + o_\delta(1) \big)\\
   & \ges  \pi^\delta_1(\eps,|y|)^2\big(\rcP[R(1)\,|\, L(1)] -  \rcP[R(1)] + o_\delta(1) \big).
   \label{eq:LcapR_lb1}
\end{align}

We will now bound the~$\rcP[R(1)\,|\, L(1)] -  \rcP[R(1)]$ from below. This is seen as a coupling argument between~$\rcP[\cdot\,|\, L(1)]$ and~$\rcP$,
with the quantity of interest being the probability of having~$R(1)$ occur for the higher configuration, but not the lower one. 
These types of arguments appeared in \cite{duminil2022planar}; we will present them briefly, stressing the specificities of our setting. 

The proof is based on three lemmas. We first state them, prove~\eqref{e.l.P(LR)-P(LR)=Delta^2_l_bd}, then prove the lemmata. 

\begin{lemma}\label{lem:iota1}
There exists a coupling~$\bbP$ between~$\rcP$ and~$\rcP[\,\cdot\,|\,L(1)]$ via a decision tree and a stopping time~$\tau$ such that, 
when~$\tau<\infty$, 
\begin{itemize}
	\item~$\calF_\tau = \delta \Z^2 \setminus e_{[\tau]}$ is a well-separated outer flower domain on~$\Lambda_{4|y|}$
	and compatible boundary conditions $(\xi,\xi')$ satisfy~$\xi < \xi'$; 
	\item the law of~$\omega$ on~$\calF_\tau$ conditionally on~$\omega_{[\tau]}$ is~$\rcP_{\calF_\tau}^{\xi}$;
	\item the law of~$\omega'$ on~$\calF_\tau$ conditionally on~$\omega'_{[\tau]}$ stochastically dominates~$\rcP_{\calF_\tau}^{\xi'}$.
\end{itemize}
Moreover,~$\bbP[\tau < \infty] \ges 1$. 
\end{lemma}	

Observe that in the above we do not claim that the law of~$\omega'$ on~$\calF_\tau$ is of the type~$\rcP_{\calF_\tau}^{\zeta}$
for some~$\zeta \geq \xi'$. Indeed, this is not always the case, but it always dominates such a law. 

\begin{lemma}\label{lem:iota2}
Fix~$\calG$  a well-separated outer flower domain on~$\Lambda_{4|y|}$ and write~$\xi < \xi'$ for its compatible boundary conditions.
There exists a coupling~$\bbP$ of~$\rcP_{\calG}^{\xi}$ and~$\rcP_{\calG}^{\xi'}$ and a stopping time~$\tau$ such that, 
when~$\tau < \infty$, 
\begin{itemize}	
	\item~$\calF_\tau = \delta \Z^2 \setminus e_{[\tau]}$ is a well-separated inner flower domain on~$\Lambda_{2|y|}(x)$;
	\item the boundary conditions induced by~$\omega'_{[\tau]}$ on~$\calF_\tau$ are a boost of those induced by~$\omega_{[\tau]}$. 
\end{itemize}
Moreover,~$\bbP[\tau < \infty] \ges \Delta^\delta(|y|,|x|)^2$. 
\end{lemma}	

\begin{lemma}\label{lem:iota3}
Fix~$\calF$ a well-separated inner flower domain on~$\Lambda_{2|y|}(x)$ and let~$\xi < \xi'$ be a boosting pair of boundary conditions.
Then 
\begin{align}\label{eq:iota3}
	 \rcP_{\calF}^{\xi'}\big[ \omega^{\xi'} \in R(1)\big] -  \rcP_{\calF}^{\xi}\big[\omega^{\xi} \in R(1) \big] \ges \pi_1^\delta(\eps , |y|)^2.
\end{align}
\end{lemma}	
In the above, we write~$\omega^{\xi} \in R(1)$ to mean that the points~$x$ and~$x+y$ are connected using~$\omega$ and potentially the wiring given by the boundary conditions.

\begin{proof}[Proof of~\eqref{e.l.P(LR)-P(LR)=Delta^2_l_bd}]
	Combining Lemmata~\ref{lem:iota2} and~\ref{lem:iota3}, we find that, 
	for any  well-separated outer flower domain~$\calG$ on~$\Lambda_{4|y|}$ with~$\xi < \xi'$ its compatible boundary conditions,
\begin{align}\label{eq:iota4}
	 \rcP_{\calG}^{\xi'}\big[ \omega^{\xi'} \in R(1)\big] - 	 \rcP_{\calG}^{\xi}\big[\omega^{\xi} \in R(1)\big] \ges \pi_1^\delta(\eps , |y|)^2 \, \Delta^\delta(|y|,|x|)^2.
\end{align}

 Now, if~$\bbP$ denotes the coupling of Lemma~\ref{lem:iota1} between~$\rcP$ and~$\rcP[\,\cdot\,|\,L(1)]$, we have 
 \begin{align}
 \rcP[R(1)\,|\, L(1)] -  \rcP[R(1)]  
 & = \bbP\big[\omega' \in R(1) \text{ but } \omega \notin R(1)\big]\\
&  \geq \bbE\Big[ \bbP\big[\omega' \in R(1) \text{ but } \omega \notin R(1)\,\big|\, \omega_{[\tau]},\omega'_{[\tau]} \big] \1_{\tau < \infty} \Big]\\
& \geq \bbE\Big[\big(\rcP_{\calF_\tau}^{\xi'}[ \omega^{\xi'} \in R(1)] - 	 \rcP_{\calF_\tau}^{\xi}[\omega^{\xi} \in R(1) ] \big)\1_{\tau < \infty}\big]\\
& \ges  \pi_1^\delta(\eps , |y|)^2 \, \Delta^\delta(|y|,|x|)^2,
 \end{align}
with the last line due to~\eqref{eq:iota4} and the fact that~$\tau < \infty$ with positive probability. Inserting this into~\eqref{eq:LcapR_lb1} we deduce~\eqref{e.l.P(LR)-P(LR)=Delta^2_l_bd}. 
\end{proof}	

We now turn to the proofs of the three lemmata. 
 \begin{figure}
	 \begin{center}
	 \includegraphics[width = 0.45\textwidth]{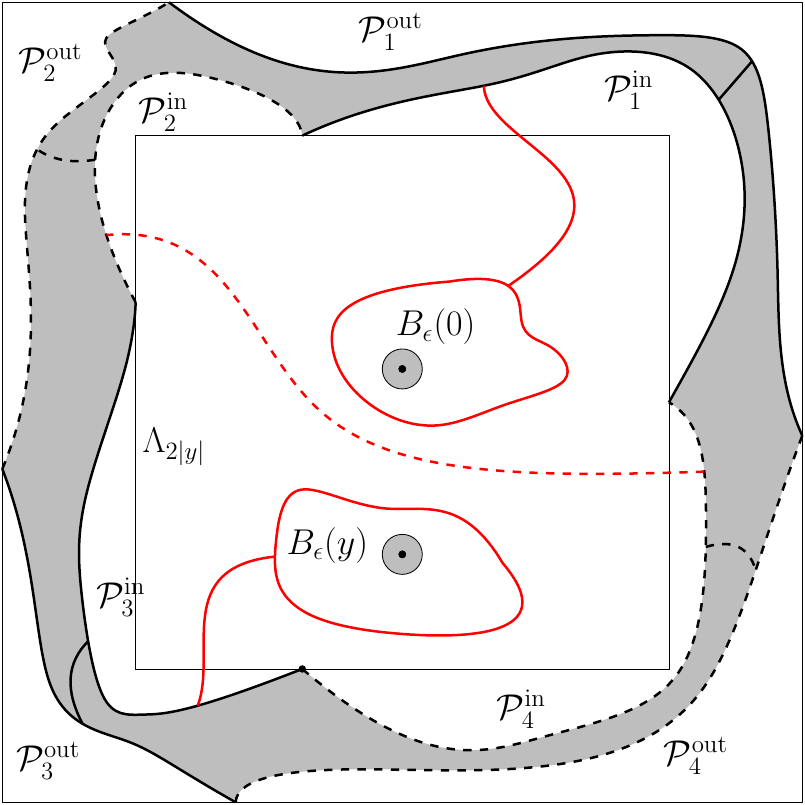}
	 \caption{The double four-petal flower domain~$(\calF_{\rm in}, \calF_{\rm out})$ between~$\Lambda_{2|y|}$ and~$\Lambda_{4|y|}$:
		$\calF_{\rm in}$ is the region surrounded by the shaded part, while~$\calF_{\rm out}$ is the one outside of the shaded part. 
		Note the connections in the shaded region between each pair of petals~$\Pin_j, \Pout_j$.\newline
		The red paths (which are part of~$\omega$) ensure the occurrence of~$H$. 
		As the grey balls~$B_\eps(0)$ and~$B_\eps(y)$ are conditioned to be connected to each other in~$\omega'$  (potentially via~$\calF_{\rm out}$), 
		they are necessarily connected in~$\omega'$ to the red circuits surrounding them, and thus to~$\Pin_1$ and~$\Pin_3$, respectively.}
	 \label{fig:iota2}
	 \end{center}
	 \end{figure}

\begin{proof}[Proof of Lemma~\ref{lem:iota1}]
	Consider the following coupling~$\bbP$ between~$\omega\sim \rcP$ and~$\omega' \sim \rcP[\cdot\,|\, L(1)]$. 
	The decision tree will only take into account the configuration~$\omega$. 
	
	First, reveal the inner flower domain~$\calF_{\rm in}$ on~$\Lambda_{2|y|}$ and the outer flower domain~$\calF_{\rm out}$ on~$\Lambda_{4|y|}$ 
	in~$\omega$. 
	Also, reveal all edges in~$\calF_{\rm in}^c \cap \calF_{\rm out}^c$. 
	We may now determine whether~$(\calF_{\rm in}, \calF_{\rm out})$ form a double four-petal flower domain in~$\omega$. 
	If not, set~$\tau = \infty$. If they do, reveal all edges in~$\calF_{\rm in}$. 
	
	Write~$H$ for the event that, in~$\omega \cap \calF_{\rm in}$,
	\begin{itemize}
		\item~$\Pin_1$ is not connected to~$\Pin_3$;
		\item there exists a circuit of open edges around~$B_\eps(0)$ that is connected to~$\Pin_1$;
		\item there exists a circuit of open edges around~$B_\eps(y)$ that is connected to~$\Pin_3$.
	\end{itemize}
	See Figure~\ref{fig:iota2} for an illustration.
	If~$H$ occurs, let~$\tau$ be the end of this stage of revealment. Otherwise set~$\tau = \infty$. 
	Using~\eqref{eq:RSW}, it is standard to show that 
	\begin{align}
		\bbP[H \,|\,(\calF_{\rm in}, \calF_{\rm out}) \text{ double four-petal flower domain}]\ges 1.
	\end{align}
	Combining this with Lemma~\ref{lem:DFPFD}, we conclude that~$\bbP[\tau < \infty] \ges 1$, as claimed. 
	
	Next we describe the laws of~$\omega$ and~$\omega'$ in the unrevealed region~$\calF_{\rm out}$ when~$\tau< \infty$. 
	Fix a realisation of~$\tau$,~$e_{[\tau]}$,~$\omega_{[\tau]}$ and~$\omega'_{[\tau]}$. 
	Write~$\xi < \xi'$ for the two compatible boundary conditions on~$\calF_{\rm out}$.
	Observe that, due to~$H$, the boundary conditions induced by~$\omega_{[\tau]}$ on~$\calF_{\rm out}$ are~$\xi$. 
	Let~$\zeta$ denote the boundary conditions induced by~$\omega_{[\tau]}'$.
	
	By the occurrence of~$H$ and the nature of~$\omega'$,~$B_\eps(0)$ is connected to~$\Pin_1$ and~$B_\eps(y)$ is connected to~$\Pin_3$ in~$\omega'_{[\tau]}$. 
	It may however be that the two are not connected to each other in~$\omega'_{[\tau]}$. 
	If~$B_\eps(0)$ is connected to~$B_\eps(y)$ in~$\omega'_{[\tau]}$, then so are~$\Pout_1$ and~$\Pout_3$, and thus~$\zeta \geq \xi'$. 
	Moreover, since~$L(1)$ is realised by~$\omega'_{[\tau]}$,~$\omega'$ in~$\calF_{\rm out}$ 
	follows the unconditioned measure~$\phi_{\calF_{\rm out}}^\zeta \geq_{\rm st} \phi_{\calF_{\rm out}}^{\xi'}$, as claimed.  
	
	Consider now the case where~$B_\epsilon(0)$ is not connected to~$B_\epsilon(y)$ in~$\omega'_{[\tau]}$.
	Write~$\calQ^{\rm out}_1$ for the set of vertices on~$\partial \calF_{\rm out}$ connected to~$B_\eps(0)$ in~$\omega'_{[\tau]}$
	and~$\calQ^{\rm out}_3$ for the those connected to~$B_\eps(y)$ in~$\omega'_{[\tau]}$.
	Due to~$H$,~$\Pout_1\subset \calQ^{\rm out}_1$ and~$\Pout_3\subset \calQ^{\rm out}_3$, 
	but since~$B_\epsilon(0)$ and~$B_\epsilon(y)$ are not connected in~$\omega'_{[\tau]}$,~$\calQ^{\rm out}_1$ and~$\calQ^{\rm out}_3$ are disjoint. 
	
	In this case, since~$L(1)$ is not realised by~$\omega'_{[\tau]}$,~$\omega'$ in~$\calF_{\rm out}$
	follows the measure~$\phi_{\calF_{\rm out}}^\zeta[\cdot\,|\, \calQ^{\rm out}_1 \lra \calQ^{\rm out}_3]$.
	Write~$\zeta'$ for the boundary condition on~$\calF_{\rm out}$ obtained from~$\zeta$ by wiring together~$\calQ^{\rm out}_1$ and~$\calQ^{\rm out}_3$. 
	Since the measures below are entirely supported on configurations with~$\calQ^{\rm out}_1$ and~$ \calQ^{\rm out}_3$ connected, we have
	\begin{align}
		\phi_{\calF_{\rm out}}^\zeta[\cdot\,|\, \calQ^{\rm out}_1 \lra \calQ^{\rm out}_3]
		= \phi_{\calF_{\rm out}}^{\zeta'}[\cdot\,|\, \calQ^{\rm out}_1 \lra \calQ^{\rm out}_3]
		\geq_{\rm st} \phi_{\calF_{\rm out}}^{\zeta'},
	\end{align}
	with the stochastic domination due to the FKG inequality. 
	Moreover~$\zeta' \geq \xi'$, as claimed. 
\end{proof}

\begin{proof}[Proof of Lemma~\ref{lem:iota2}]
	A very similar statement was proved in \cite[Lemma~5.3]{duminil2022planar}. 
	The proof adapts readily to our setting. See also \cite[Remark~5.5]{duminil2022planar}. 
\end{proof}

 \begin{figure}
	 \begin{center}
	 \includegraphics[width = 0.4\textwidth]{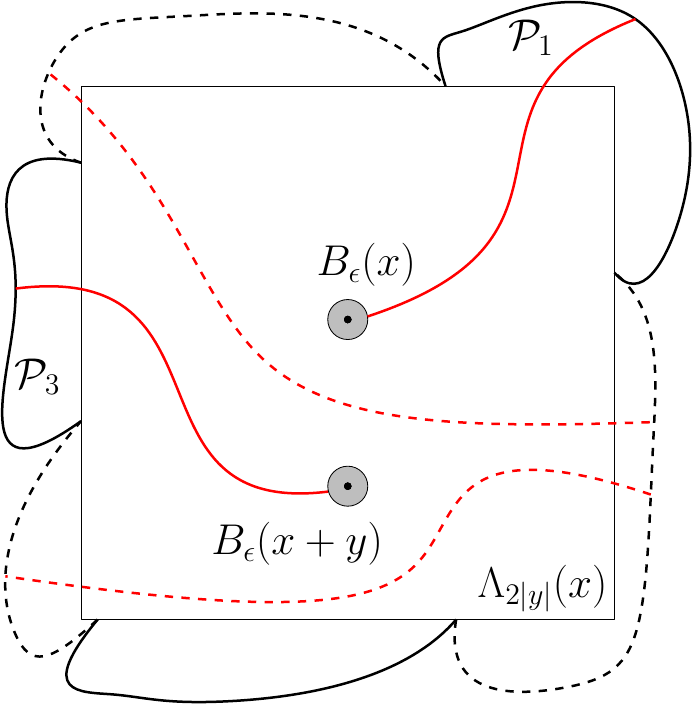}
	 \caption{The flower domain~$\calF$ contains six petals. The red paths (part of~$\omega$) ensure the occurrence of~$H$. 
	 By dually connecting all dual petals, we ensure that~$B_\eps(x)$ and~$B_\eps(x+y)$ are not connected in~$\omega^\xi$.
	 All paths may be constructed in disjoint ``tubes''. By~\eqref{eq:RSW}, the dual paths occur with constant probability, while the primal ones each occur with probability~$\pi_1^\delta(\eps , |y|)$.}
	\label{fig:iota3}
	\end{center}
 \end{figure}

\begin{proof}[Proof of Lemma~\ref{lem:iota3}]
	Fix a flower domain~$\calF$ as in the statement and~$(\xi, \xi')$ a boosting pair of boundary conditions. 
	By monotonicity, it suffices to consider the case where~$\xi$ and~$\xi'$ are compatible with~$\calF$ and differ by the wiring of two petals which we denote~$\Pin_1$ and~$\Pin_3$. 
	
	Consider an arbitrary increasing coupling~$\bbP$ of~$\omega\sim \rcP_{\calG}^{\xi}$  and~$\omega' \sim \rcP_{\calG}^{\xi'}$.
	Let~$H$ be the event that
	~$B_\eps(x)$ is connected to~$\Pin_1$,~$B_\eps(x+y)$ is connected to~$\Pin_3$, but~$B_\eps(x)$ is not connected to~$B_\eps(x+y)$; see Figure~\ref{fig:iota3} for an illustration. 
	 Standard applications of~\eqref{eq:RSW} show that 
	 \begin{align}
	 \bbP[ \omega \in H] = \rcP_{\calF}^{\xi}[H] \ges  \pi_1^\delta(\eps , |y|)^2. 
	 \end{align}
	 
	 Finally, when~$H$ occurs, we have~$\omega^{\xi'} \in R(1)$ but~$\omega^{\xi} \notin R(1)$. Thus, 
	\begin{align}
	 \rcP_{\calF}^{\xi'}\big[ \omega^{\xi'} \in R(1)\big] -  \rcP_{\calF}^{\xi}\big[\omega^{\xi} \in R(1) \big] 
	 &= \bbP\big[(\omega')^{\xi'} \in R(1) \text{ but } \omega^{\xi} \notin R(1)\big]\\
	&\geq \bbP\big[\omega^{\xi'} \in R(1) \text{ but } \omega^{\xi} \notin R(1)\big]
	\ges \pi_1^\delta(\eps , |y|)^2,
	\end{align}
	as claimed. 
\end{proof}

	\small
\bibliographystyle{abbrv}
\bibliography{ref}

\end{document}